\documentclass[a4paper,10pt,twoside]{article}   
\usepackage{promat1}

\usepackage{amsmath,amsmath,amsthm,amsfonts,amssymb}
\usepackage[all]{xy}
\usepackage{verbatim}

\usepackage[final]{pdfpages}

\usepackage{enumitem}
\setitemize[0]{leftmargin=0.58cm,labelsep=5pt}

\usepackage{titlesec}
\titleformat{\section}{\Large\normalfont\bfseries}{\thesection.}{0.5em}{}

\newcommand{\fd}{\bf}

\newcommand{\cabezaprincipal}{{\small \textit{Pro Mathematica, XXXI, 62 (2021), 
61-93, ISSN 2305-2430}}}
 \numberwithin{equation}{section}
\usepackage{todonotes,csquotes,url}
\usepackage[sort,comma,numbers]{natbib}
\usepackage{graphicx}
\usepackage{subcaption}
\usepackage{caption}
\usepackage{float}

\newtheorem{theorem}{Theorem}[section]
\newtheorem{lemma}[theorem]{Lemma}

\newtheorem{corollary}[theorem]{Corollary}

\theoremstyle{definition}

\newtheorem{example}[theorem]{Example}

\def\R{{\mathbb R}}
\def\Z{{\mathbb Z}}

\newcommand{\eop}{\hfill$\square$}

\begin{document}

\includepdf[pages=-]{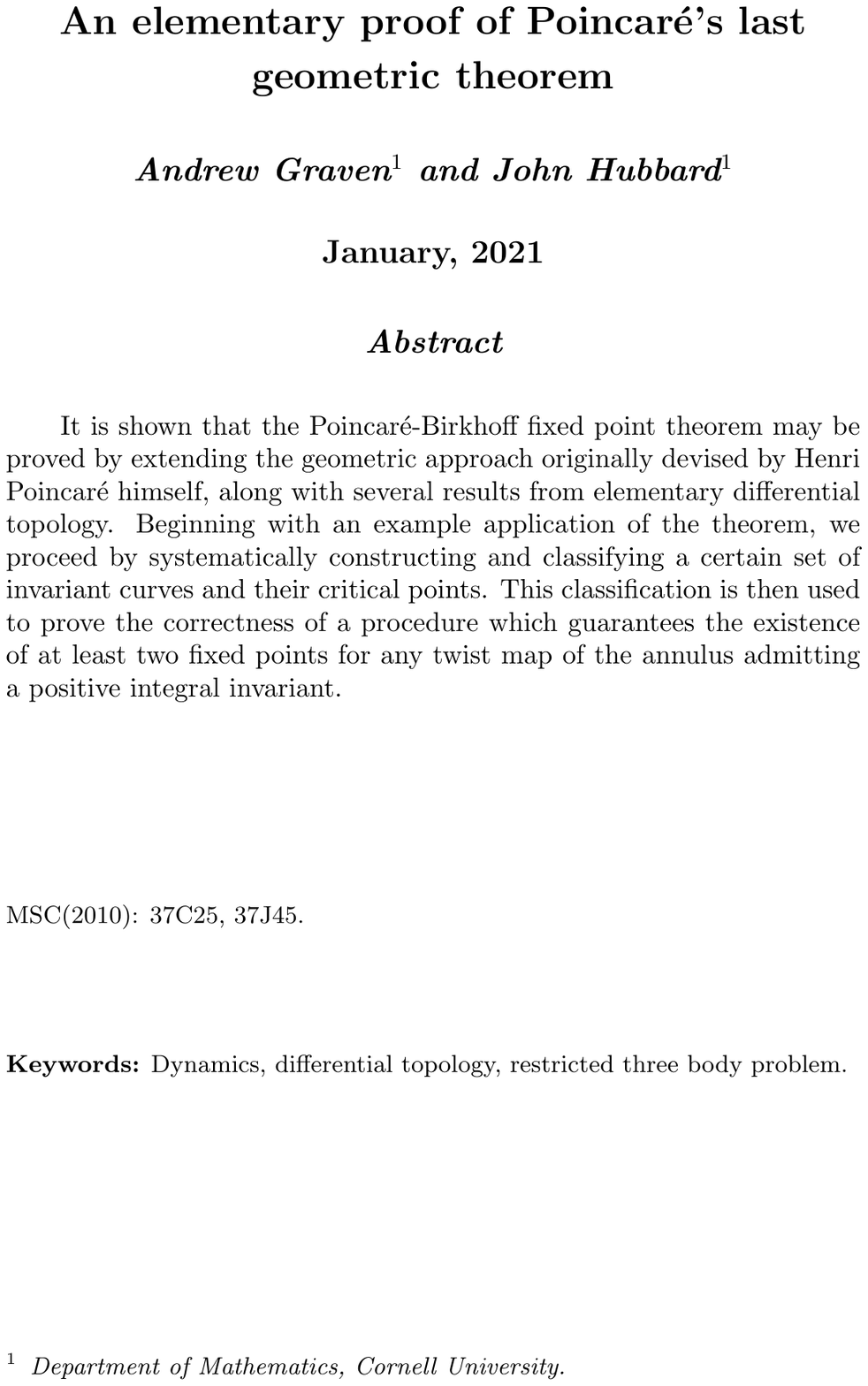}

\setcounter{page}{62}% NUMERO DE PAGINA MAS 1, CON EL QUE INICIA EL ARTÍCULO

\markboth {\small {\it Andrew Graven and John Hubbard}}{\small {\it An elementary proof of Poincar\'e's last geometric  theorem}}

%%%%%%%%%%%%%%%%%%%%%%%%%%%%%%%%%%%%%%%%%%%%%%%%
%%%%%%%%%%%%%
%%%%                                  MAIN BODY 
%%%%%%%%%%%%%%
%%%%%%%%%%%%%%%%%%%%%%%%%%%%%%%%%%%%%%%%%%%%%%%%%5

\begin{displayquote} 
``J'ai demontr\'e il y a longtemps d\'ej\`a, 
l'existence des solutions p\'eriodiques du probl\`eme des trois corps; 
le r\'esultat laissait cependant encore a d\'esirer; 
car, si l'existence de chaque sorte de solution \'etait \'etablie pour les petites valeurs des masses, 
on ne voyait pas ce qui devait arriver pour des valeurs plus grandes, 
quelles \'etaient celles de ces solutions qui subsistaient et dans quel ordre elles disparaissaient. 
En r\'efl\'echissant a cette question, je me suis assur\'e que la reponse devait d\'epen-dre 
de l'exactitude ou de la fausset\'e d'un certain th\'eor\`eme de geom\'etrie dont l'\'enonce est tr\`es simple, 
du moins dans le cas du probl\`eme restreint et des probl\`emes de Dynamique o\`u il n'y a que deux degr\'es de libert\'e.''
\footnote{I demonstrated long ago the existence of periodic solutions to the three-body problem; 
however, the result still left something to be desired; 
for, if the existence of each class of solution was established for small values of the masses, it was not clear what happens for larger values: which of these solutions remained and in what order they disappeared. Reflecting on this question, 
I convince myself that the answer depends 
on the truth or untruth of a certain geometric theorem whose statement is very simple, 
at least in the case of the restricted problem and problems of dynamics 
in which there are only two degrees of freedom.} 
 
{} \hfill Henri Poincar\'e, 1912 \cite{PoincareProof}
\end{displayquote}

\begin{section}{Introduction}
In this paper, 
we present an elementary proof of the Poincar\'e-Birkhoff fixed point theorem, 
otherwise known as ``Poincar\'e's last geometric theorem''. 
The theorem roughly states that any measure-preserving diffeomorphism of the annulus which twists 
the inner and outer boundaries 
in opposite directions has at least two fixed points. 
Poincar\'e originally conjectured this in his 1912 paper ``Sur un Th\'eor\'eme de G\'eom\'etrie'' \cite{PoincareProof}, 
in which he presents an elegant geometric proof in several special cases. 
However he did not succeed in proving the theorem in general.

It was not until after Poincar\'e's death that George David Birkhoff published 
the first ostensibly complete proof in 1913 \cite{birkhoff1913}. 
Unfortunately, Birkhoff's argument for the existence of the second fixed point relied on 
a fallacious application of the Poincar\'e theorem \cite{PoincareTheoremTopo}, 
known in its more general case as the Poincar\'e-Hopf index theorem 
(see Guillemin and Pollack \cite[page 134]{G&P}). 
In particular, the Poincar\'e theorem implies that 
the indices of the fixed points of $f$ must sum to zero. 
Thus, if $f$ has at least one fixed point, and this fixed point is of non-zero index, 
then there must exist at least one additional fixed point. 
However, this neglects the possibility of the fixed point having index zero. 
Birkhoff ultimately presented a correct proof of the general case of the theorem in his 1926 paper 
``An Extension of Poincar\'e’s Last Geometric Theorem'' \cite{birkhoff1926}, 
taking an analytic approach distinct from that of Poincar\'e. We show that, by applying several elementary results in differential topology, 
one can prove the general case of the theorem, 
including the existence of the second fixed point, along the lines of Poincar\'e's original argument.
\end{section}

\begin{section}{Statement of the theorem}

Let $A=\R/\Z \times [0,1]$ be the standard annulus, with $x$ the $1-$periodic coordinate and $y$ the radial coordinate, 
with universal cover $\widetilde A=\R\times [0,1]$. 
Let $f:A \to A$ be a $C^1$ diffeomorphism mapping each boundary component to itself, and  
$\widetilde f:\widetilde A \to \widetilde A$ be a lift to the universal cover.

Write $\widetilde f(x,y)= (\tilde{f}_1(x,y), \tilde{f}_2(x,y))$.  
The map $\widetilde f$ is called a {\fd twist map} if the two conditions,
$\tilde{f}_1(x,0)-\tilde{f}_1(0,0)<x$ and $\tilde{f}_1(x,1)-\tilde{f}_1(0,1)>x$ are satisfied for all $x\in \R$.      
This is independent of the choice of lift and, 
as a consequence of periodicity, only needs to be checked for $x\in [0,1]$. 
We also call $f$ a {\bf twist map} if $\tilde{f}$ is a twist map. 

In addition, we say $f$ {\bf admits a positive integral invariant} 
if there exists a function $d\mu:A\rightarrow\mathbb{R}$ such that $d\mu>0$ for almost every $x\in A$ and 
the measure associated with $d\mu$, namely $\mu(U)=\int_Ud\mu$, 
satisfies $\mu(U)=\mu(f(U))$ for all measurable $U\subseteq A$.

\begin{theorem}\label{thm:PLGT}
If $f:A\rightarrow A$ is a twist map admitting a positive integral invariant, then $f$ has at least two fixed points.
\end{theorem}

\begin{corollary}[Poincar\'e's Last Geometric Theorem \cite{PoincareProof}]\label{thm:PLGTCorr1}
If $f$ is an area preserving twist map, then $f$ has at least two fixed points.\eop
\end{corollary}

\end{section}

\begin{section}{Why we care} 

Poincar\'e was interested in this result in order to prove the existence of periodic motions in the restricted 3-body problem 
\cite{PoincareProof,PoincareReview,BB}. 
The example of the forced pendulum below is exactly this sort of problem, 
and we find that the result nicely fills in   
the KAM theorem picture, at least for systems with two degrees of freedom 
\cite{KAM}.  

\begin{example} \rm 
The equation we will use for the forced pendulum is 
\[
x''+\sin x = a \cos t.
\]
This may not look like a Hamiltonian system, but it is if we add a variable $s$ conjugate to $t$, leading to
\[
H\begin{pmatrix}x\\y\end{pmatrix}= \frac {y^2}2-\cos x-ax\cos t +s,
\]
\begin{equation}\label{pendulumhamiltoniansystem}
\begin{split}
x'&=\frac{\partial H}{\partial y}= y \\
y'&=-\frac{\partial H}{\partial x}= -\sin x+a \cos t \\
\end{split}
\qquad
\begin{split}
t'&=\frac{\partial H}{\partial s}= 1 \\
s'&= -\frac{\partial H}{\partial t}= ax\sin t.
\end{split}
\end{equation}

The period map $P:\R^2 \to \R^2$ is given by $P(x,y)=\varphi_{2\pi}(x,y)$, where $\varphi_t$ is the time $t$ flow of the vector field
\[
\begin{pmatrix} x'\\y'\end{pmatrix}= \begin{pmatrix} y\\ -\sin x + a\cos t \end {pmatrix}.
\]
The function $P$ is an area preserving map of the plane, either because the vector field has vanishing divergence, or because the $(x,y)$-plane is a Poincaré section for the Hamiltonian system (\ref{pendulumhamiltoniansystem}). 

\begin{figure}[ht!]
\centering
\begin{subfigure}[b]{0.475\linewidth}    %%% 0.45
   \includegraphics[width=\linewidth] {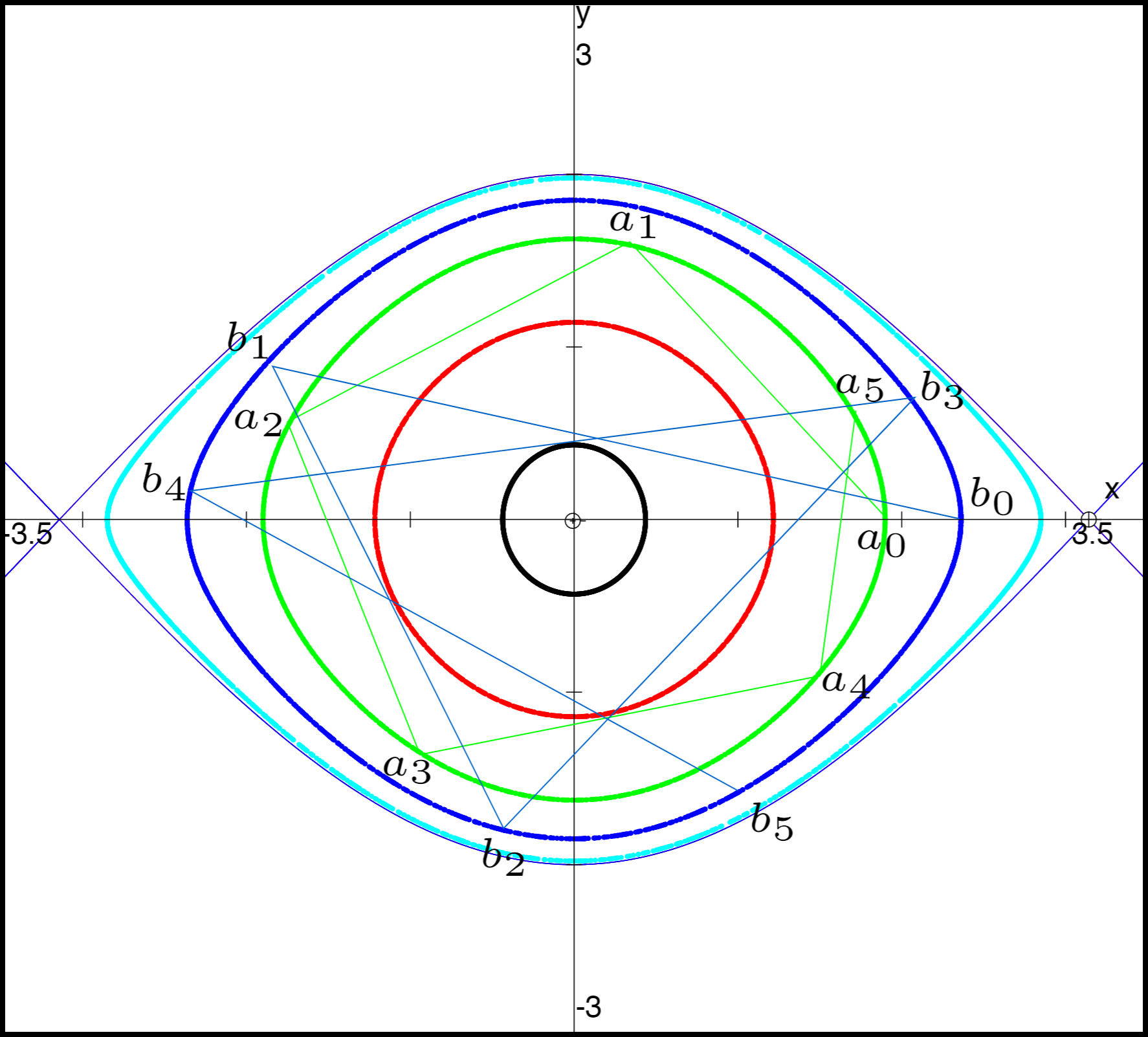}
   \caption{\sl no forcing with orbits}
   \end{subfigure}
   \begin{subfigure}[b]{0.475\linewidth}   %%%% 0.45
   \includegraphics[width=\linewidth] {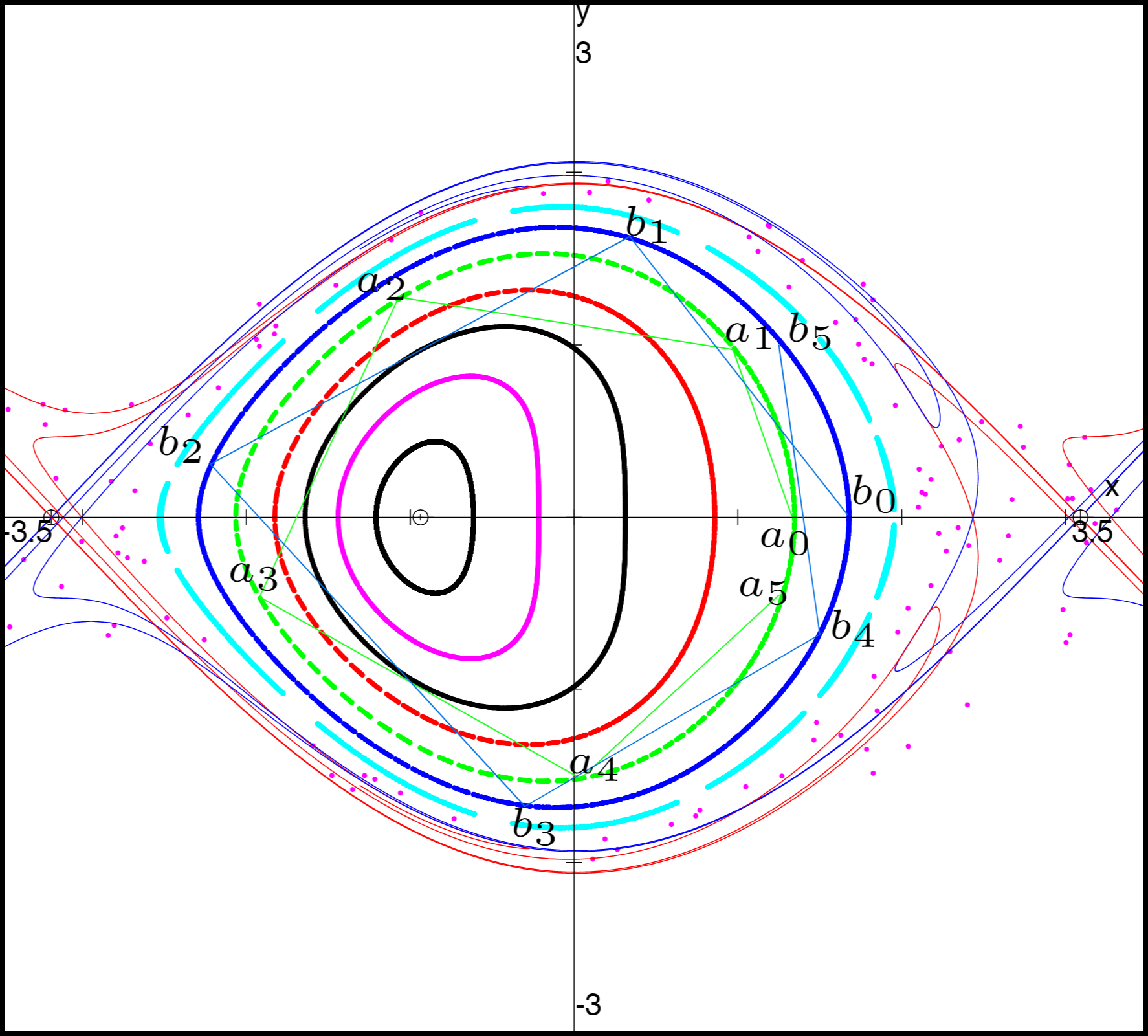}
   \caption{\sl forcing with orbits}
   \end{subfigure}

   \caption{\sl Figure 1(a) shows several orbits for the unforced equation with $a=0$,    
   together with the separatrices of the saddle corresponding to the upward unstable equilibrium of the pendulum. 
   We have drawn the first six   
   iterates of two points, connecting them by lines, 
   and we have labeled the points $a_0,\dots, a_5$ and $b_1,\dots, b_5$.  
   Figure 1(b) is the corresponding picture for $a=.1$.  
   The ``equilibrium'' is now a periodic orbit with period $2\pi$, hence a fixed point of $P$. 
   It is ``surrounded'' by invariant curves, really linear windings on  2-dimensional tori.  }
   \label{fig:orbits}
\end{figure}  

When $a=0$, we can easily find a Lyapunov function for the reduced system. 
In fact, the standard Lyapunov function for a time independent Hamiltonian system is given by
\[
V\begin{pmatrix} x\\y\end{pmatrix}= \frac {y^2}2 -\cos x. 
\]
In this case, the Lyapunov function $V$ is constant with respect to the flow, so the orbits of $P$ lie on level curves of $V$. 
Figure \ref{fig:orbits}(a) shows various orbits, 
with emphasis on the first five iterates of two of the points. 
The orbits of points with irrational rotation number are dense in simple closed curves, 
so a randomly chosen orbit will almost surely be of this sort. 

Figure \ref{fig:orbits}(b) shows the $(x,y)$-plane when $a=.1$.  
The implicit function theorem says that fixed points of $P$ 
corresponding to the stable equilibrium $\mathbf 0$ of the pendulum 
must exist for $a$ sufficiently small,
as well as fixed points corresponding to the unstable equilibrium, 
and that these should still have stable and unstable manifolds. 
The number $a=.1$ is sufficiently small, 
and we see these equilibria, 
as well as the stable and unstable manifolds, 
which have become much more complicated since they now intersect transversally.\bigskip

\begin{figure}[ht!]
\centering
\begin{subfigure}[b]{0.4752\linewidth}   %%%% 0.452 
   \includegraphics[width=\linewidth] {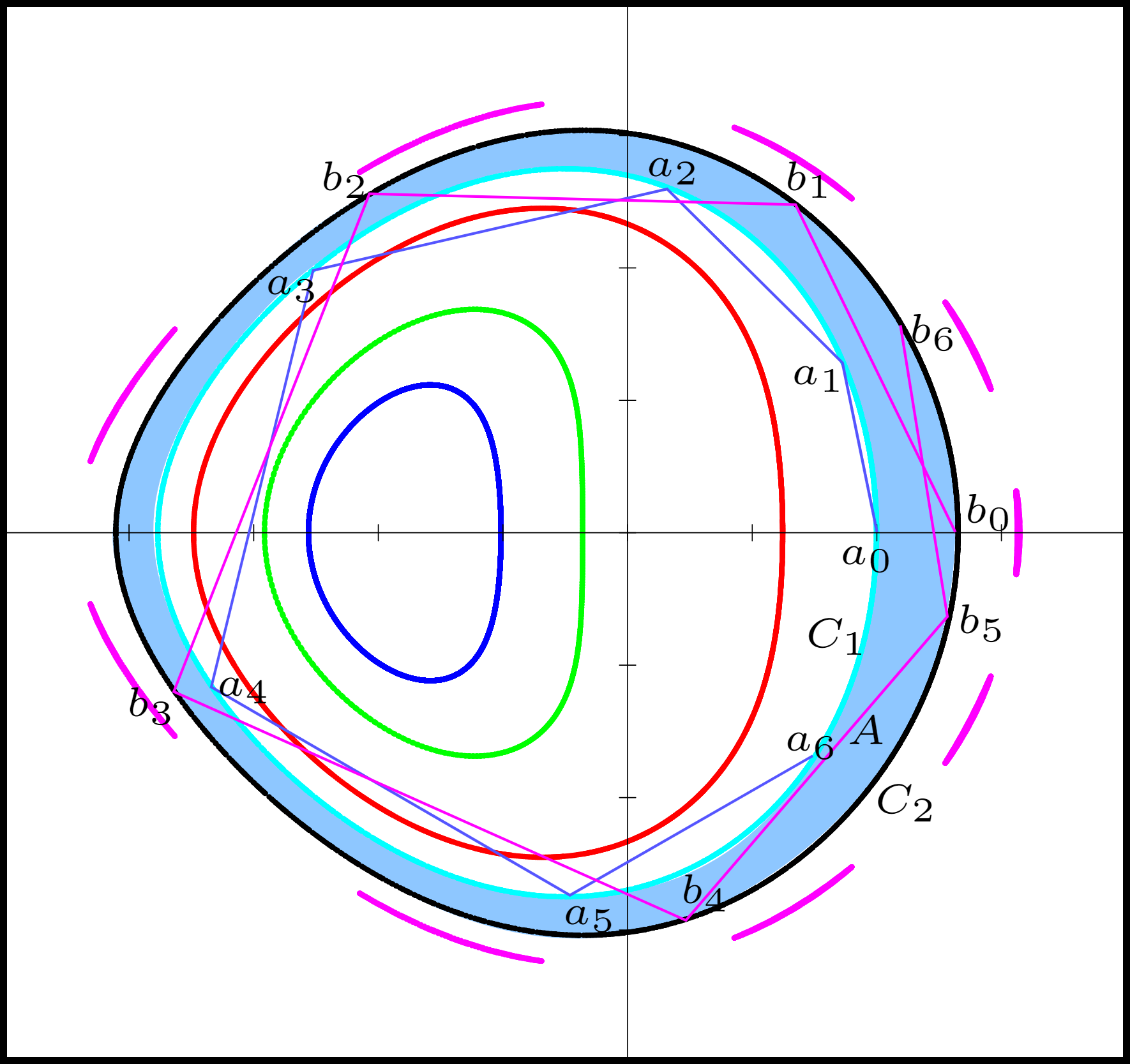}
   \caption{\sl The annulus $A$ bounded by  $C_1$ and $C_2$}
   \end{subfigure}
   \begin{subfigure}[b]{0.475\linewidth}    %%%%% 0.45 
   \includegraphics[width=\linewidth] {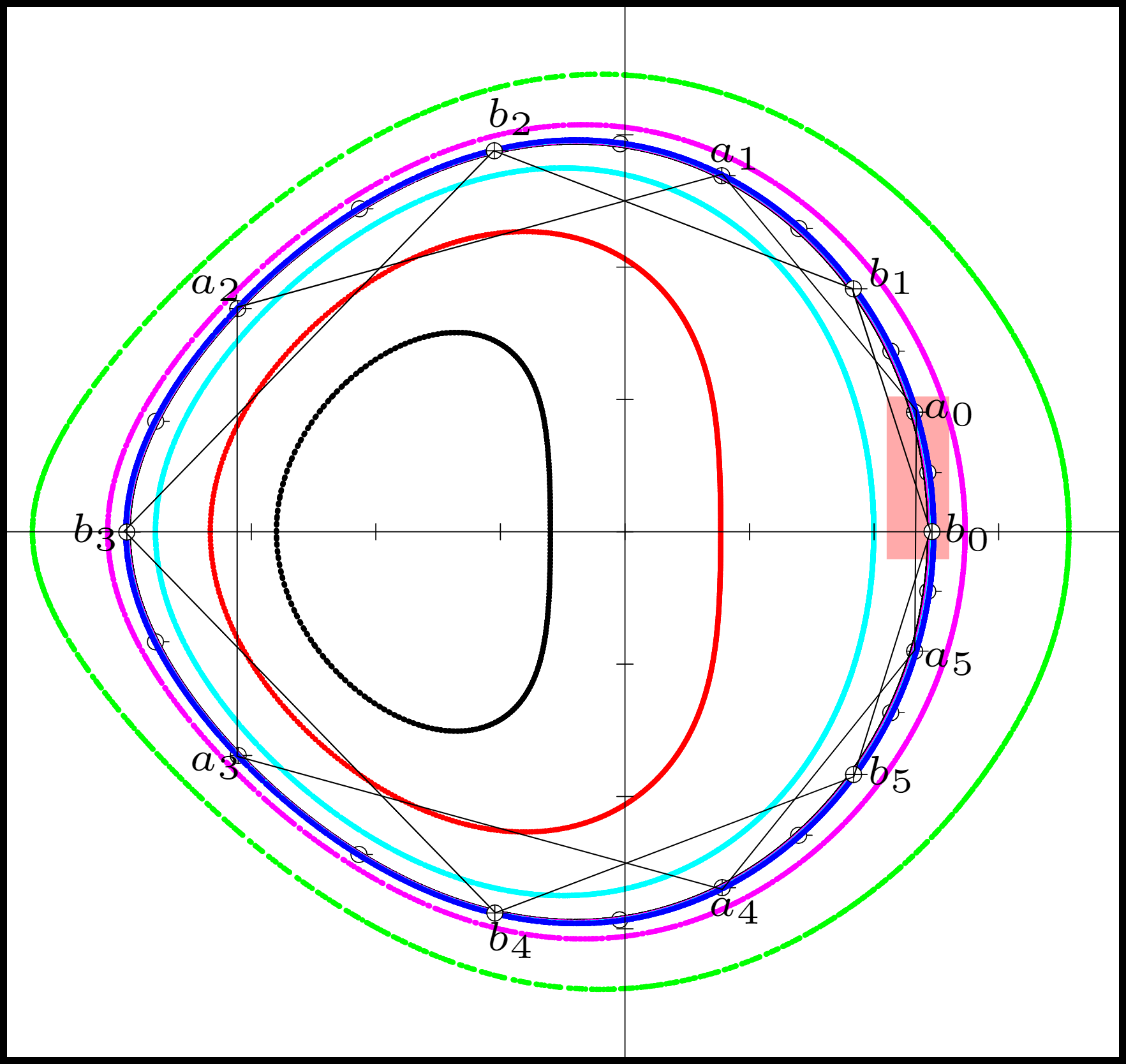}
   \caption{\sl Two six-cycles of saddles and two six-cycles of centers}     
   \end{subfigure}
   \caption{\sl In Figure 2(a), the orbit of $a_0$ on $C_1$ indicates that its rotation number is a bit less than $1/6$. 
   On the other hand, 
   the orbit of $b_0$ on $C_2$ indicates that its rotation number is a bit more than $1/6$. 
   Thus $P^{\circ 6}:A \to A$ is a twist-map. 
   Figure 2(b) shows the cycles of period 6 found by Newton's method. 
   There are two 6-cycles of saddles and two 6-cycles of centers. 
   The saddles have separatrices, of course, and 
   the KAM theorem guarantees that there are invariant cycles of curves around the centers. }
   \label{fig:periodsix}
\end{figure}  

The KAM theorem guarantees that close to the stable equilibrium for the unperturbed system,  
for rotation numbers sufficiently irrational and close enough to the rotation number of the stable equilibrium, 
and for $a$ small, orbits dense in a simple closed curve and with the same rotation number 
will exist for the perturbed system \cite{KAM}. It is very hard to quantify the ``sufficientlies'' in KAM theorem, 
and the numbers one could get from the proof are presumably absurdly pessimistic.
 
In  Figure \ref{fig:periodsix}, 
we will assume that the cyan curve $C_1$ passing through $a_0=(1,0)$ is the closure of the orbit of $a_0$, 
and the black curve $C_2$ passing  through the point $b_0=(1.3,0)$ is the closure of the orbit of $b_0$; 
the computer seems to indicate this is the case. 
We have drawn in the first seven points of the orbits of $a_0$ and $b_0$. 
The orbits of these points indicate that the rotation number of $P:C_1 \to C_1$ is smaller than $1/6$, 
and that the rotation number of $P:C_2 \to C_2$ is greater than $1/6$. 

Thus the region between $C_1$ and $C_2$ inclusive is an annulus, $A$, and $P^{\circ 6}:A \to A$ is a twist map. 
Theorem \ref{thm:PLGT} asserts that $P^{\circ 6}$ has at least two fixed points in $A$; 
but $P^{\circ 6}$ is a 6th iterate, so there must be two periodic cycles of length 6 in $A$. 
We will look for them by Newton's method (for $P^{\circ 6}$, itself a solution of a periodic differential equation!). 
The ``band'' corresponding to period 6 is sufficiently narrow that it is hard to see what is happening there. 
Figure \ref{fig:periodsixblowup} shows a blowup of both this region and of the similar island of period $3$, 
which isn't so flattened and thus we can see more of its structure.

\begin{figure}[ht!]
\centering
\begin{subfigure}[b]{0.45\linewidth}
   \includegraphics[width=\linewidth] {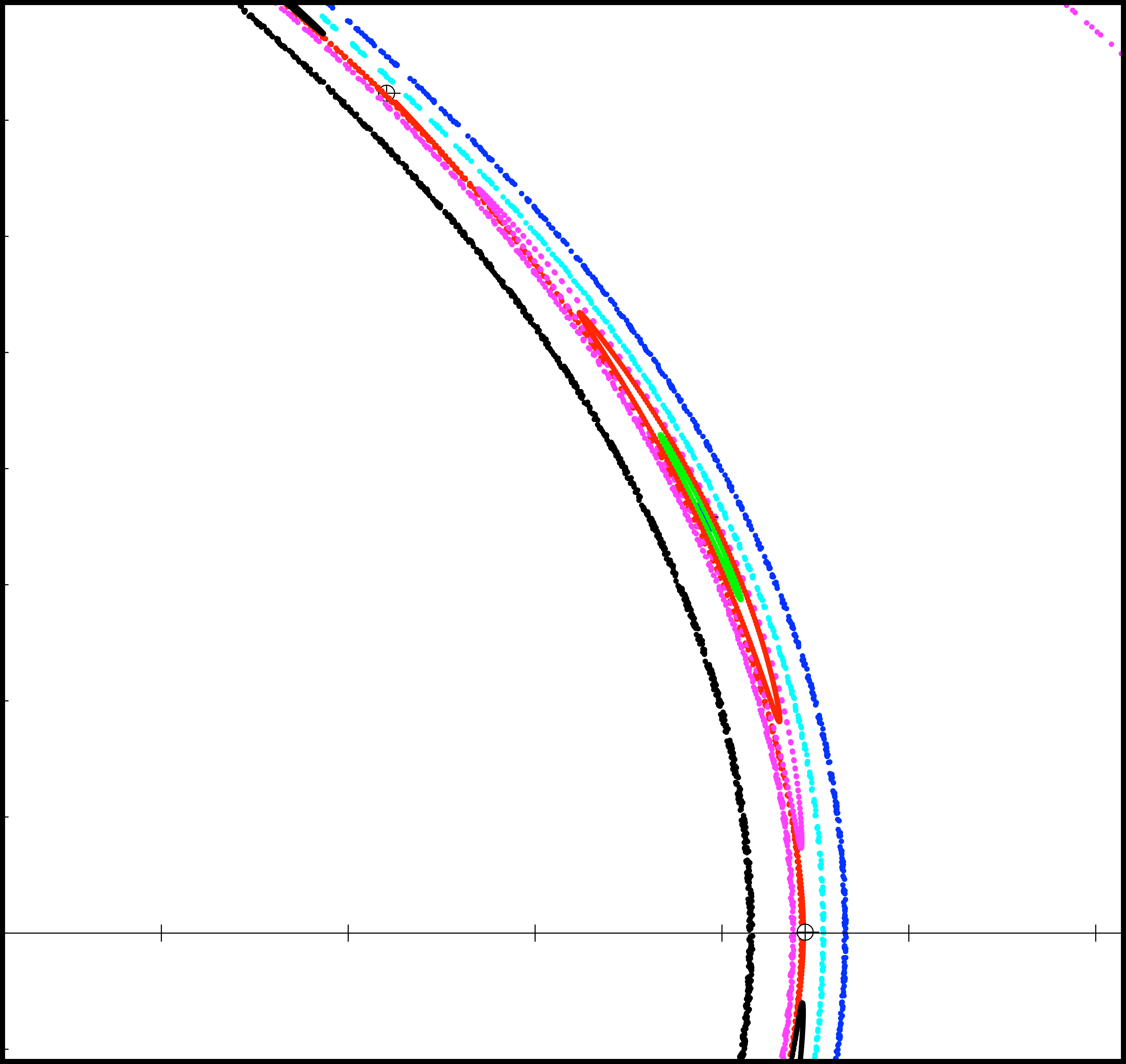}
   \caption{\sl The annulus $A$ bounded by  $C_1$ and $C_2$}
   \end{subfigure}
   \begin{subfigure}[b]{0.45\linewidth}
   \includegraphics[width=\linewidth] {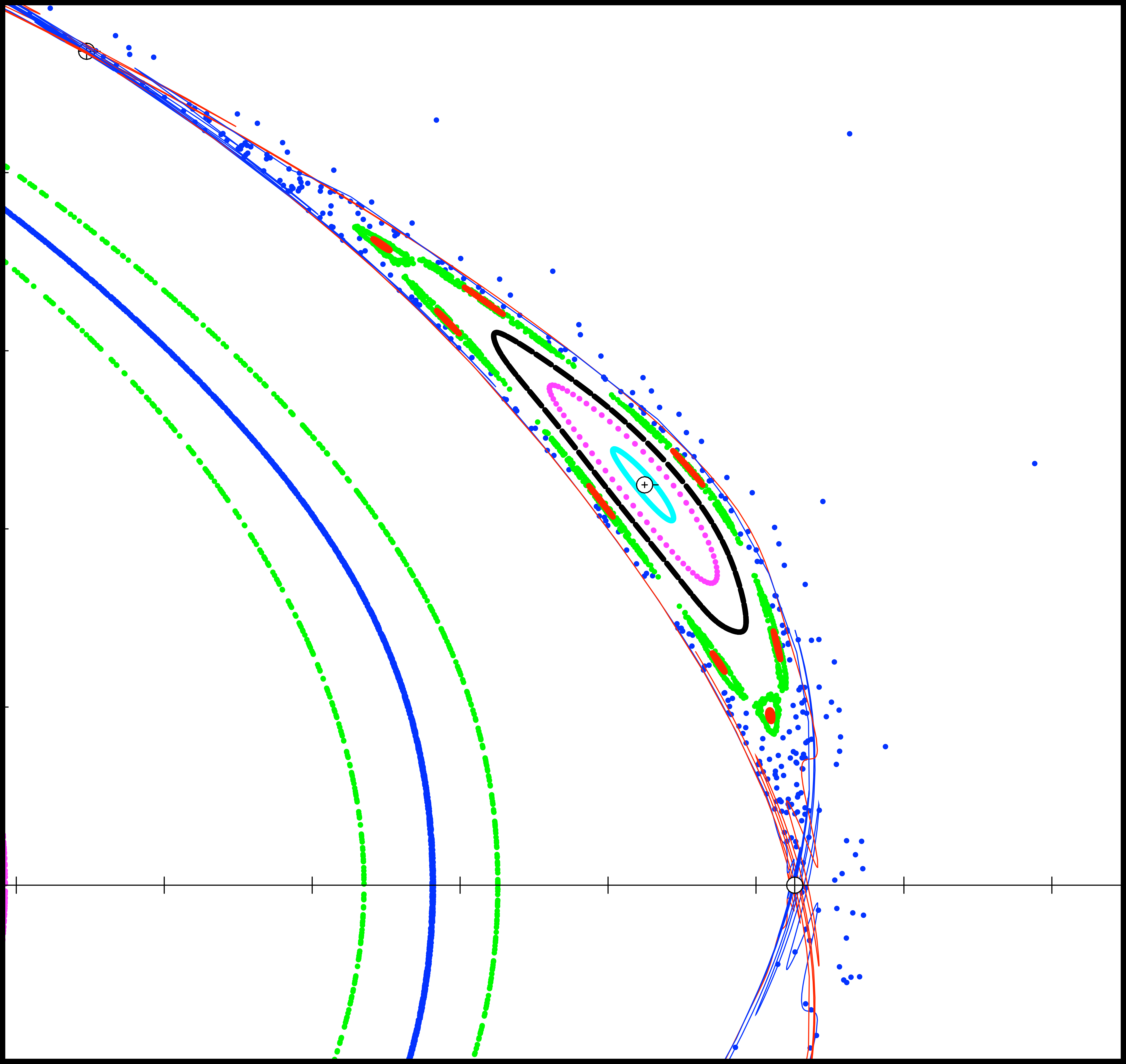}
   \caption{\sl two six-cycles of saddles and two six-cycles of centers}    
   \end{subfigure}
       \caption{\sl Figure 3(a) shows a blowup of the region around one of the centers of period six. 
     In Figure 3(b), there is a blowup of the region around a center of period $3$, a bit further out in the picture. 
     Here we can clearly see invariant curves surrounding the center 
     (really cycles of three invariant curves), and islands. 
     The pattern continues {\it ad infinitum.}}
   \label{fig:periodsixblowup}
\end{figure} 

In Figure \ref{fig:periodsixblowup} we notice that the center of period 3 is 
the organizing center of a region rather like the large-scale region represented on the right in Figure \ref{fig:orbits}. 
There are two saddles whose stable and unstable manifolds form a {\bf homoclinic tangle}; 
the stable and unstable manifolds intersect transversally, 
and it follows that they accumulate on each other.

 \begin{figure}[ht!]
\centering
\begin{subfigure}[b]{0.451\linewidth}
   \includegraphics[width=\linewidth] {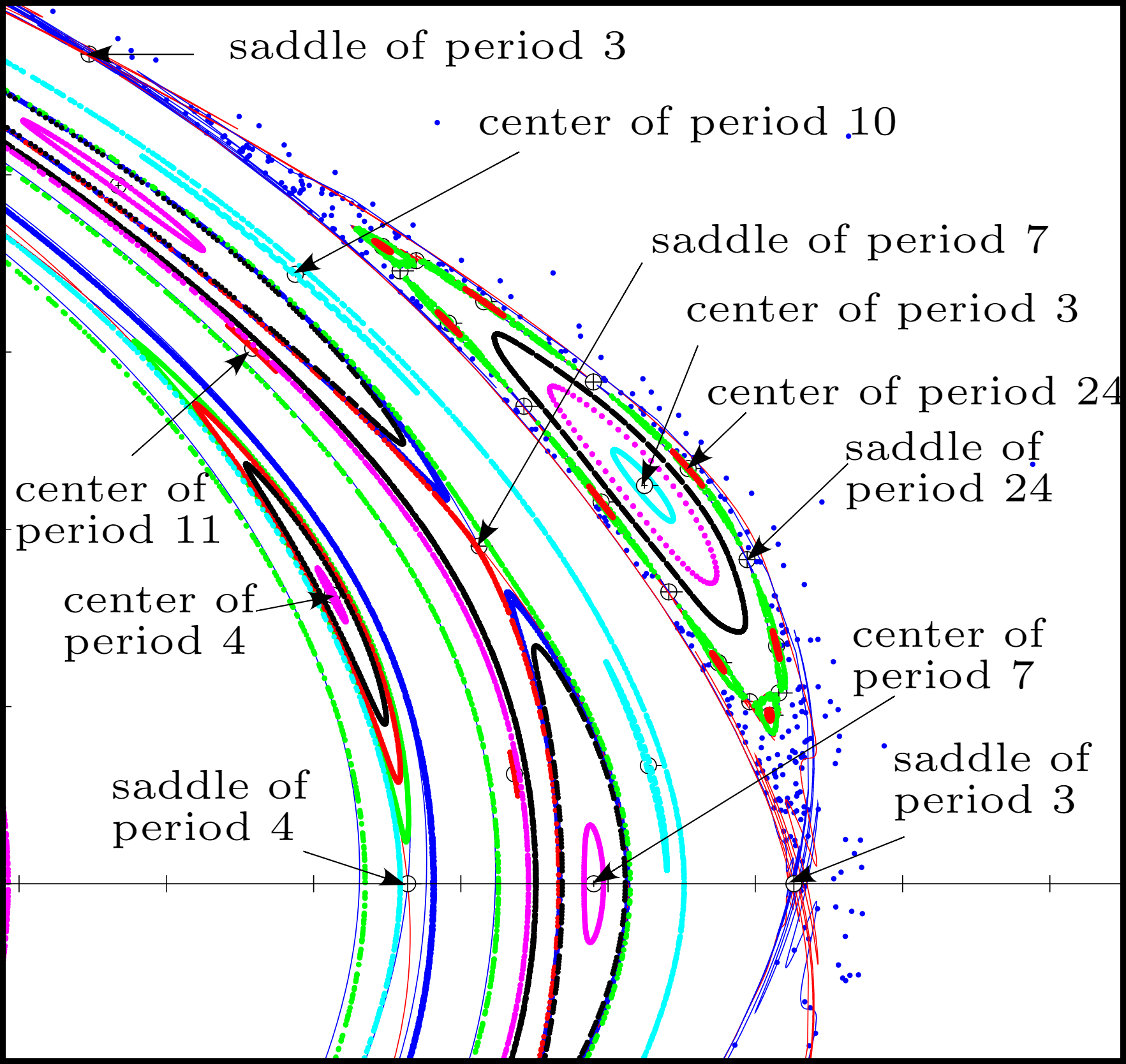}
   \caption{\sl The annulus $A$ bounded by  $C_1$ and $C_2$}
   \end{subfigure}
   \quad
   \begin{subfigure}[b]{0.45\linewidth}
   \includegraphics[width=\linewidth] {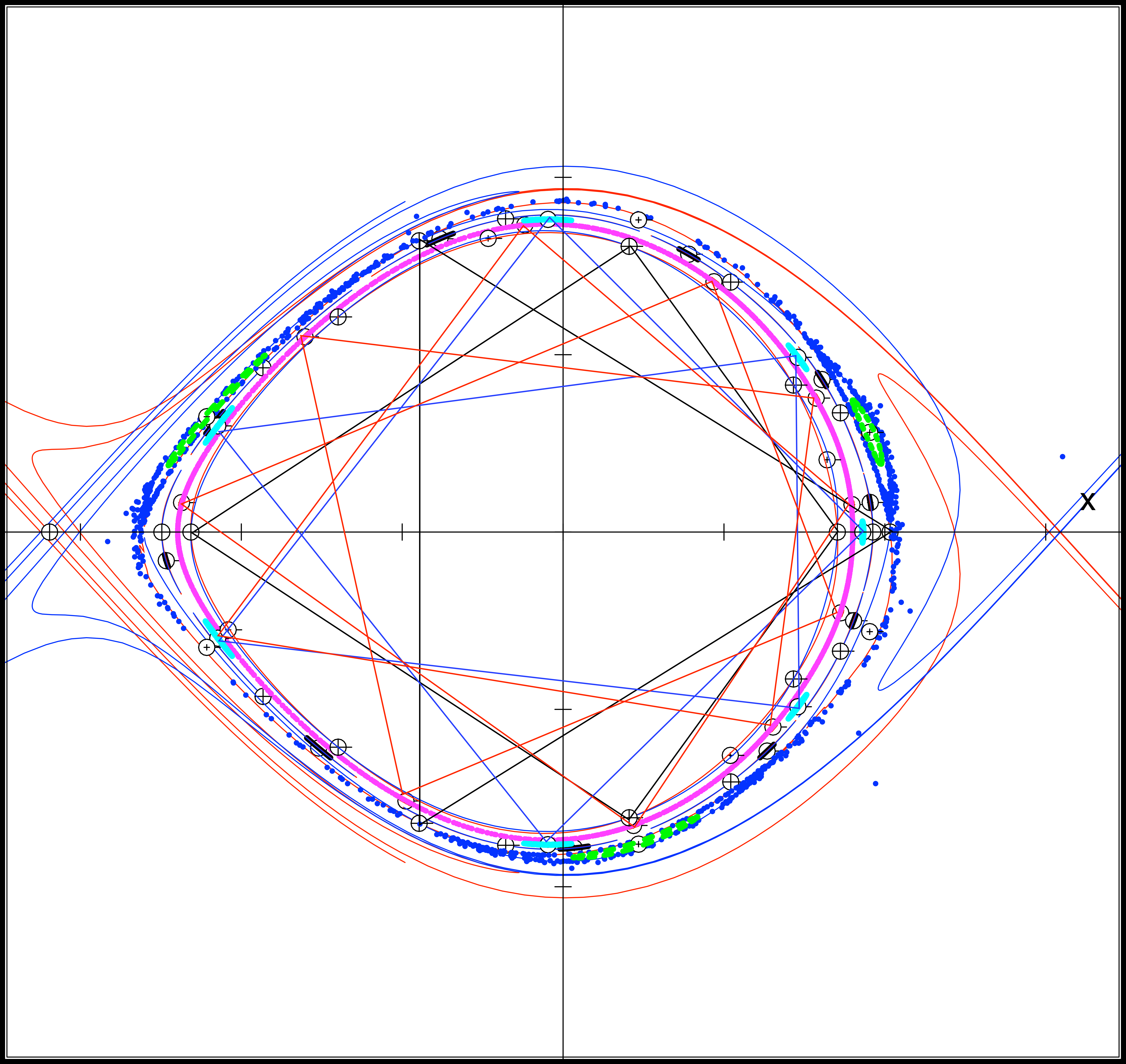}
   \caption{\sl Two six-cycles of saddles and two six-cycles of centers}
   \end{subfigure}

\caption{\sl The left figure shows the orbits of period 3, 
with two  satellite  orbits of period 24, two orbits of period 4 (one of centers and one of saddles), 
two orbits of period 7 (also one of centers and one of saddles), and one orbit of period 10 and one of period 11. 
On the right, we see some explanation of the periods from above. The ``Farey sum'' of $1/3$ and $1/4$ is $2/7$, and thus the orbit of period 7 (blue lines in the picture) forms a $2/7$ star. The Farey sum of $1/3$ and $2/7$ is $3/10$ and, indeed, the orbit of of period $10$ forms a $3/10$ star. Similarly for $1/4$ and $2/7$ giving a $3/11$ star.}
   \label{fig:Fareytreeinpendulum}
\end{figure} 

A bit of further reflection will show that the pattern they form is much more complicated than that: 
there are eight saddles forming two 4-cycles. 
Each has a stable and unstable manifold, and all of these accumulate on each other. 
There is a family of invariant curves (really 3-cycles of invariant curves)   
surrounding the period 3 center, 
but they also exist for sufficiently irrational rotation numbers, 
leaving regions corresponding to the rational numbers where we can apply Theorem \ref{thm:PLGT}, 
giving cycles of periods that are multiples of 3.  
In the picture, we see such a periodic cycle of period 24, surrounded by its own invariant curves. 
This pattern of ``order'', {\it i.e.,} motion on invariant curves, 
technically called {\bf quasi-periodic} motion, 
occurs with positive probability, but with dense subsets of ``chaos'' regions containing homoclinic tangles, 
but also further regions of order, and so forth, {\it ad infinitum...}

\end{example}
\end{section}

\begin{section}{Proof of Theorem \ref{thm:PLGT}}

We shall begin by presenting the geometric intuition for the proof. 
Then we will proceed to prove the existence of the first fixed point and, finally, 
extend that argument to show the existence of the second.

\subsection{Intuition for the proof}
\begin{figure}[ht!]
\centering
\includegraphics[width=\textwidth]{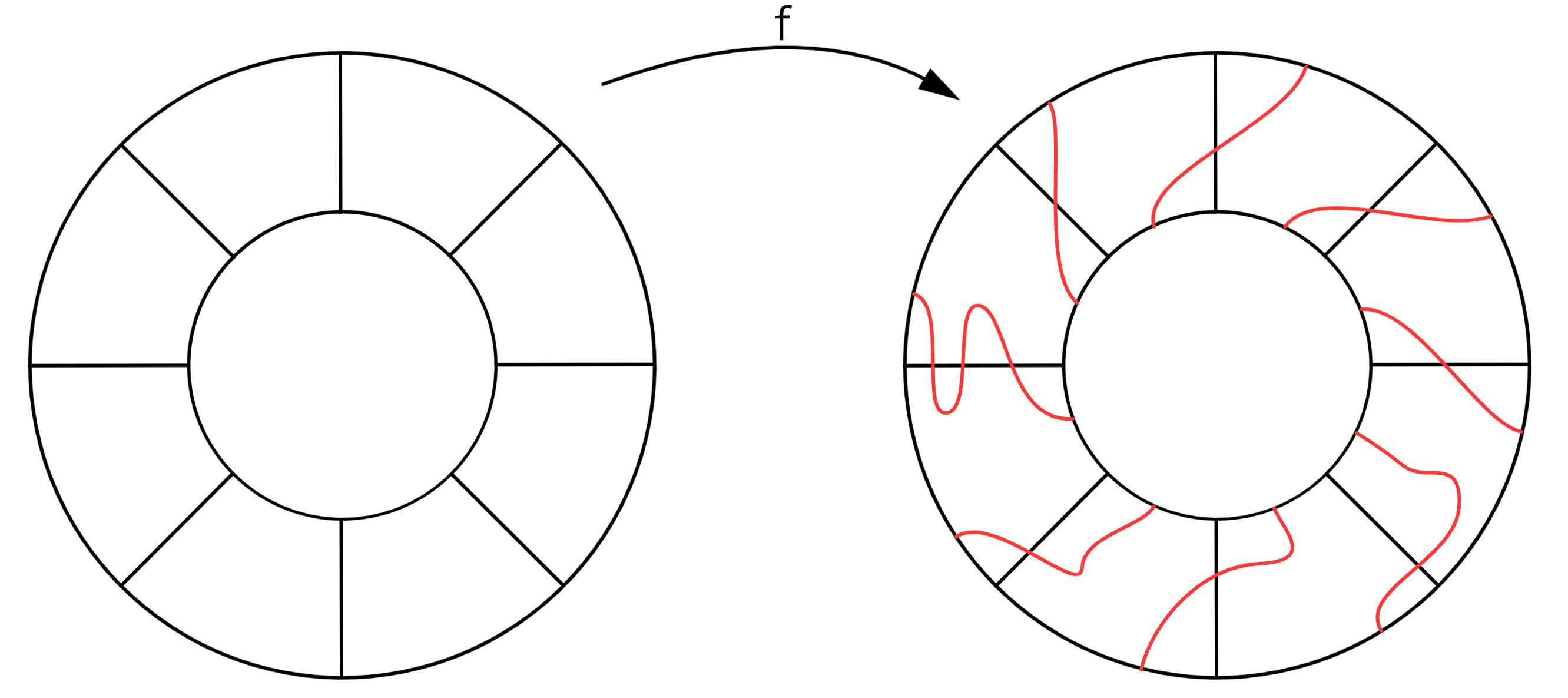}
\caption{\sl The fibres of $\alpha$ (black) and their images (red).}
\label{fig:Intuition1}
\end{figure}

Let us begin by considering the fibres of $\alpha(x,y)=x$ in $A$, which correspond to radial lines. 
These fibres necessarily intersect their images 
as a consequence of the twist condition and the intermediate value theorem, as can be seen in Figure \ref{fig:Intuition1}.

Next, isolate the intersections of these fibres with their images 
as in  Figure 6(a). %%%% \ref{fig:Intuition2a}.  KEEP THIS COMMENT !!!!!!!!!!!!!!!!!!!
One can imagine that if we were to run this process for progressively denser sets of fibres of $\alpha$, 
we might hope to obtain a set of smooth curves. 
These curves would be $x-$invariant in the sense that they satisfy the equation $f_1(x,y)=x$ 
({\it i.e.,} the $x-$coordinate of the points in these curves is unchanged by $f$).
\bigskip

\begin{figure}[ht!]
\centering
\begin{subfigure}[b]{0.44\linewidth}
\includegraphics[width=\linewidth]{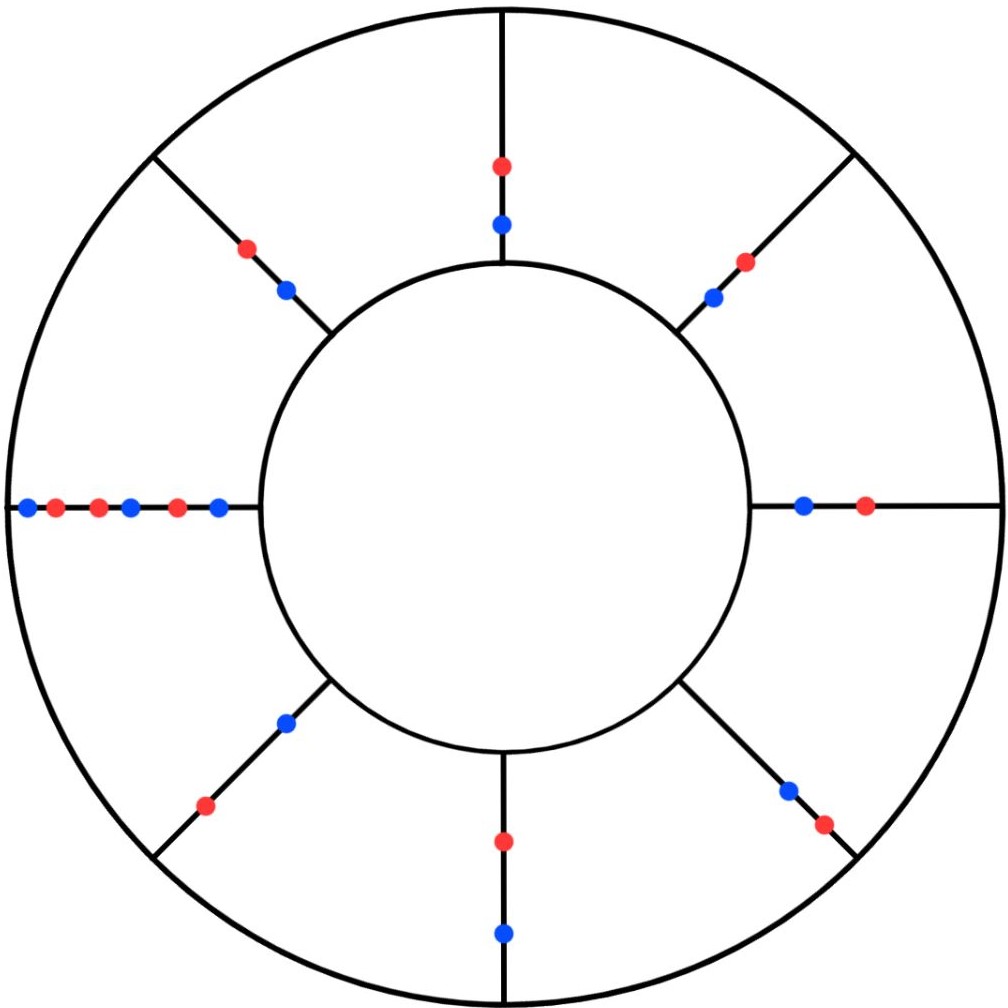}
\caption{\sl Fibre intersection points (red) and their preimages (blue).\\}\label{fig:Intuition2a} 
\end{subfigure}
\hfill
\begin{subfigure}[b]{0.44\linewidth}
\includegraphics[width=\linewidth]{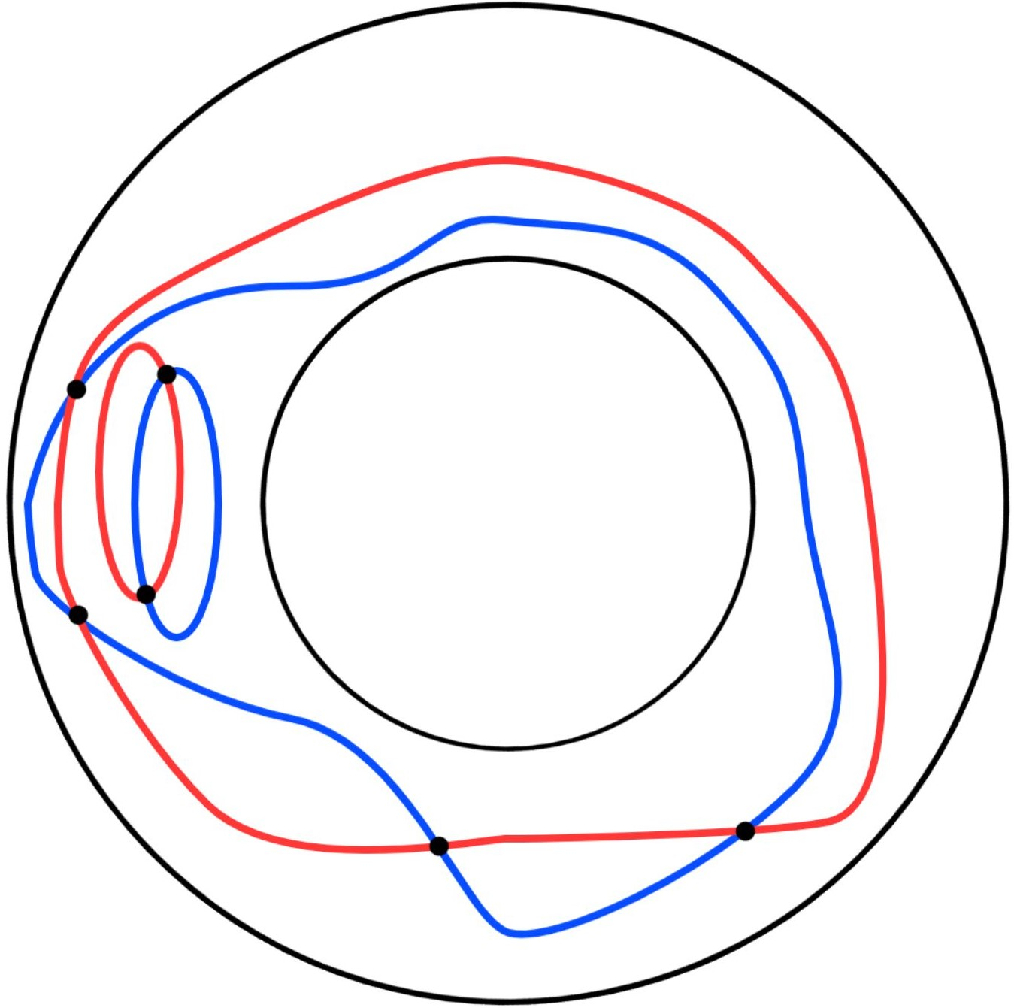} 
\caption{\sl Invariant curves (blue), their images (red), together with their intersections (black).}\label{fig:Intuition2b}
\end{subfigure}
\caption{\sl Construction of the $x-$invariant curves.}
\label{fig:Intuition2}
\end{figure}

The intersection points in Figure 6(b) %%%%% \ref{fig:Intuition2b}       KEEP !!!!!
must exist 
because otherwise the region enclosed by some $x-$invariant curve would necessarily 
be mapped to a proper subset or superset of itself, 
which contradicts the assumption that $f$ admits a positive integral invariant. 
Then one might conclude that these intersection points are fixed points. 
However, there are two claims here needing further justification:

\begin{enumerate}
    \item The invariant curves exist.
    \item The intersection points of invariant curves and their images constitute fixed points of $f$.
\end{enumerate}
It turns out (1) is indeed provable, however (2) is not necessarily true. 
Thus, the proof will proceed as a variation of the intuitive approach outlined here.

\subsection{Existence of invariant curves}
Proving the existence of the hinted at $x-$invariant curves is nearly a direct application of the preimage theorem to the map
$$
F(x,y)=\alpha(f(x,y))-\alpha(x,y)=f_1(x,y)-x.
$$
In which case $F^{-1}(0)$ would be composed of the desired set of invariant curves. 

\begin{theorem}[The Preimage Theorem \cite{G&P}, page 21] \label{thm:Preimage}   
Let $y\in Y$ be a regular value of a $C^1$ map $F:X\rightarrow Y$;  
 so, in particular, $DF(x)$ is surjective for all $x$ such that $F(x)=y$. 
 Then $F^{-1}(y)$ is a submanifold of $X$  of dimension  $\text{dim}(F^{-1}(y))=\text{dim}(X)-\text{dim}(Y)$. 
\eop
\end{theorem}

However, it cannot be guaranteed that $0$ is a regular value of our $F$,  
a necessary condition for the 
application of the preimage theorem. 
Thus, we need a slightly stronger result, the parametric transversality theorem.

\begin{theorem}[The Parametric Transversality Theorem \cite{G&P}, page 68]\label{thm:PTT}
Consider a $C^1$ map
$F:S\times X \rightarrow Y$ such that only $X$ has boundary. 
Take $Z$ to be a submanifold of $Y$, also without boundary.
If $F$ is transversal to $Z$, then 
almost every $F_s:X\rightarrow Y$ in the one parameter family $\{F_s\}_{s\in S}$ is transversal to $Z$.\eop
\end{theorem}

To apply Theorem \ref{thm:PTT} we parameterize $\alpha$ by an additional parameter, $\phi\in\mathbb{R}_{+}$, 
by $\alpha_\phi(x,y)=\phi \cdot x-y$. 
This naturally implies a redefinition of $F:\mathbb{R}_+\times A\rightarrow\mathbb{R}$, 
given now by 
$$
F(\phi,x,y)=\alpha_\phi(f(x,y))-\alpha_\phi(x,y). 
$$
When we wish $\phi$ to be held constant, we denote $F(\phi,x,y)$ by $F_\phi(x,y)$. 

Now we need to show $0$ is a regular value of $F$. In particular, we need to show that
$$
DF=
\left[f_1-x,
\phi(\dfrac{\partial f_1}{\partial x}-1)-\dfrac{\partial f_2}{\partial x},
1-\dfrac{\partial f_2}{\partial y}+\phi\dfrac{\partial f_1}{\partial y}
\right]
$$
is onto for all $(x,y)\in F^{-1}(0)$. If $DF$ weren't onto, 
then each of its entries must be equal to $0$, which would imply $x=f_1$. 
This along with $F_\phi(x,y)=0$ would tell us $f(x,y)=(x,y)$. 
That is, $DF$ can only fail to be onto if $(x,y)$ is a fixed point of $f$. 

Thus, there are exactly two possibilities:

\begin{enumerate}
\item There is a set of $\phi$ of full measure such that $F_\phi$ is transverse to $0$.
\item There is a fixed point of $f$ in $A$.
\end{enumerate}

After covering several additional preliminary results, we begin the central thrust of the proof of Theorem \ref{thm:PLGT} in Section 5. There we show that (1) contradicts the existence of a positive integral invariant, thus establishing that $f$ has at least one fixed point in $A$. Then in Section 6 we show that the assumption in (2) of a unique fixed point also leads to a contradiction. In fact, by a slight modification of the previous argument the assumption of a unique fixed point still contradicts the existence of a positive integral invariant.  

If (1) holds, then, by the definition of transversality, there exists $\phi\in\mathbb{R}_+$ such that $0$ is a regular value of $F_\phi$. 
So, by the preimage theorem, $F^{-1}_\phi(0)$ is a compact $1-$dimensional submanifold of $A$ and thus 
contains a finite number of connected components 
(all closed curves, by the classification of compact $1-$manifolds). 
Give each component the preimage orientation. 
This choice of orientation endows each component which 
wraps around the annulus with winding number $\pm1$ and 
all other components with winding number $0$. 
Figure \ref{fig:ComplexEx} provides an example of the sort of geometry we might expect here.
When there is no risk of ambiguity, $F^{-1}_\phi(0)$ is referred to as $\mathcal{I}_\phi$ or simply as $\mathcal{I}$.

\begin{figure}[H]
\centering
\includegraphics[width=\textwidth]{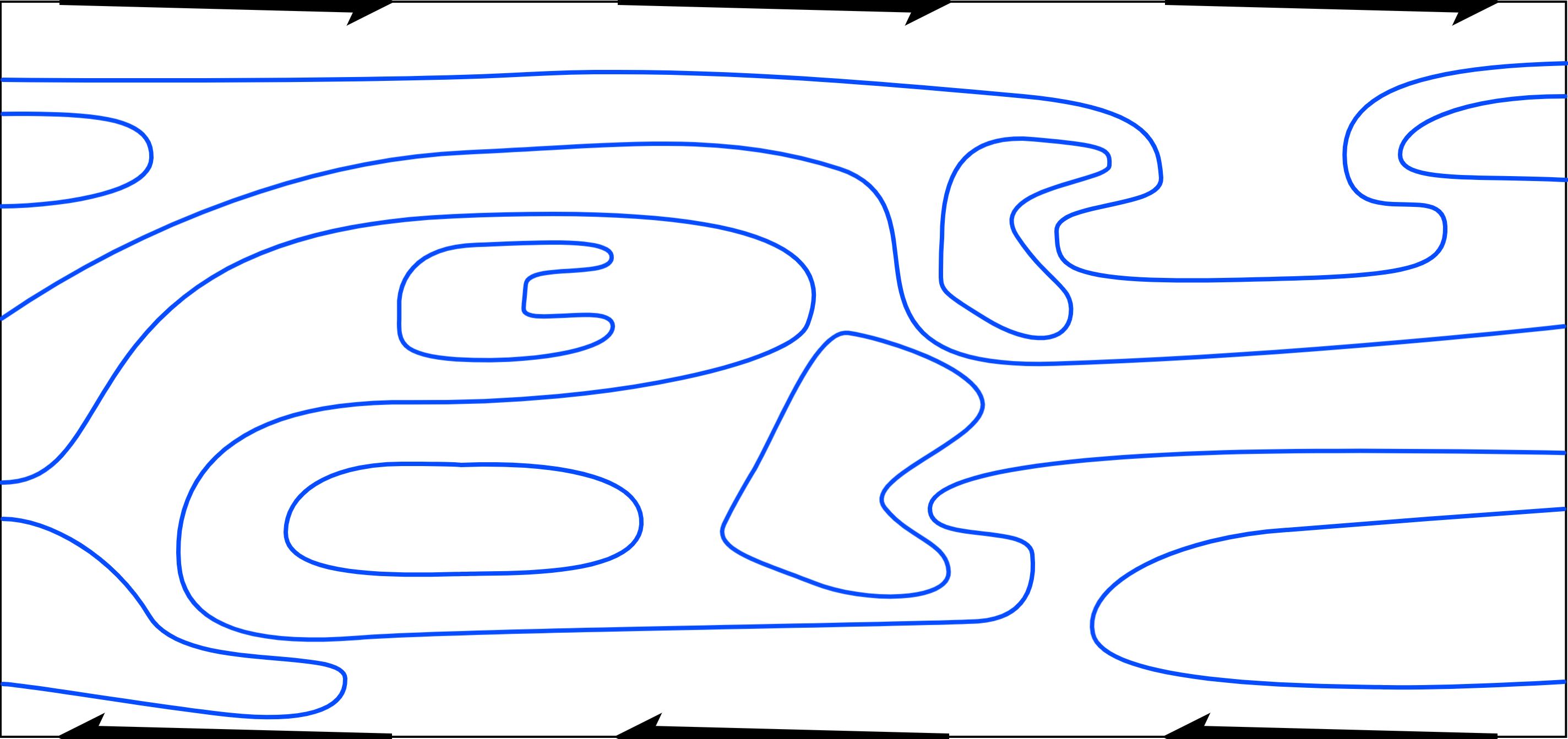}
\caption{\sl A possible instance of $\mathcal{I}$, with nine  
components, three of which have non-zero winding number. 
The arrows denote the twisting of the inner and outer boundaries in opposite directions. 
Throughout the paper, it will be standard to represent the annulus by the $1$-periodic strip, as shown here.}
\label{fig:ComplexEx}
\end{figure}

\subsection{A modification of $f$}
The parameterization of $\alpha$ with respect to $\phi$, 
necessary to apply Theorem \ref{thm:PTT} and the preimage theorem,   
complicates things slightly 
by endowing $\alpha_\phi(x,y)=\phi \cdot x-y$ with slanted fibres 
in comparison to the straight fibres of the original $\alpha(x,y)=x$. 
So, to simplify the analysis, we conjugate the system by the linear transformation
$$
T(x,y)=\left(x-\dfrac{y}{\phi}, y\right)
$$
with inverse 
$$
T^{-1}(x,y)=\left( x+\dfrac{y}{\phi}, y\right).
$$  
In these coordinates, the fibres of $\alpha$ are straight, 
so $f$ translates $\mathcal{I}$ vertically (viewing $A$ as the $1$-periodic strip). 
Moreover, conjugation by $T$ is simply a change of coordinates, so the hypotheses of the theorem remain satisfied. 

\subsection{Components with winding number $\pm1$ within $\mathbf{\mathcal{I}}$}

Let $\gamma:[0,1]\rightarrow\mathbb{R}^2$ be a parameterization of a simple closed  curve. 
Then the {\bf winding number} of $\gamma$ about $z\in\mathbb{R}^2$ is given by
$$W(\gamma,z)=\deg\left(\dfrac{\gamma(x)-z}{|\gamma(x)-z|}\right)$$
where $\deg(g)$ is the topological degree of $g$.

\begin{lemma}[An integral formula for the  
winding number \cite{G&P}]\label{lemma:StokesWinding}
The winding number $W(\gamma)=W(\gamma,0)$ of a curve $\gamma$ about the origin can be computed as
$$
W(\gamma)=\oint_\gamma d\theta,
$$
where $\theta(\vec{x})$ is the ``argument'' or ``angle'' of $\vec{x}$.  
In Cartesian coordinates 
for example we have $\theta(x,y)=\tan^{-1}\left(\dfrac{y}{x}\right)$.\eop
\end{lemma}

\begin{lemma} 
\label{lemma:NzWinding}
The winding number of $\mathcal{I}$ is equal to $1$. 
In other words, we have $W(\mathcal{I})=1$.
\end{lemma}
\begin{proof}
Let $\gamma$ be a simple curve in $A$ transversal to $\mathcal{I}$ which runs from one boundary to the other as 
$\gamma(0)\in \mathbb{R}/\mathbb{Z}\times\{1\}$ and $\gamma(1)\in\mathbb{R}/\mathbb{Z}\times\{0\}$. 
Then by the twist property, we have $F(\gamma(0))>0>F(\gamma(1))$. 
Thus, $F\circ\gamma$ changes sign an odd number of times on $[0,1]$, 
implying it has an odd number of zeros. 
Therefore $\gamma([0,1])$ intersects $\partial A\cup\mathcal{I}$ an odd number of times with alternating orientation, 
so $W(\partial A\cup\mathcal{I})=1$. Then, by Stokes' theorem, we obtain
$$W(\partial A)=\oint_{\partial A}d\theta=\int_{A}d^2\theta=\int_A0=0,$$
which immediately implies
$$
1=W(\partial A\cup\mathcal{I})=W(\partial A)+W(\mathcal{I})=W(\mathcal{I}).
$$
\end{proof}
\end{section}

\begin{section}{The existence of a fixed point} 

At this point, one may like to make an argument along the following lines. 
Take an $S\subseteq \mathcal{I}$ satisfying $W(S)=0$. 
Then $S$ intersects $f(S)$ at least twice because otherwise one set would bound the other, 
violating the existence of a positive integral invariant. 
These intersections are fixed points, so we are done.

Unfortunately, this argument fails on its last line. 
If $S$ takes on multiple values of $y$ for the same value of $x$, 
it is possible for the intersection points of $S$ and $f(S)$ to occur at different $y$ values for the same $x$. 
Moreover, for components with winding number $0$, intersections cannot be guaranteed.

Using Figure \ref{fig:CounterEx} as a starting point, 
we will show that any such ``counterexample'' to the theorem admits a curve which violates the existence of a positive integral invariant. 
Our approach will need to be general enough to handle complicated geometries, like that in Figure \ref{fig:ComplexEx}. 
Thus, to simplify analysis, we make the following observation which will constrain the range of possibilities.

\begin{figure}[H]
\centering
\includegraphics[width=1.0\textwidth]{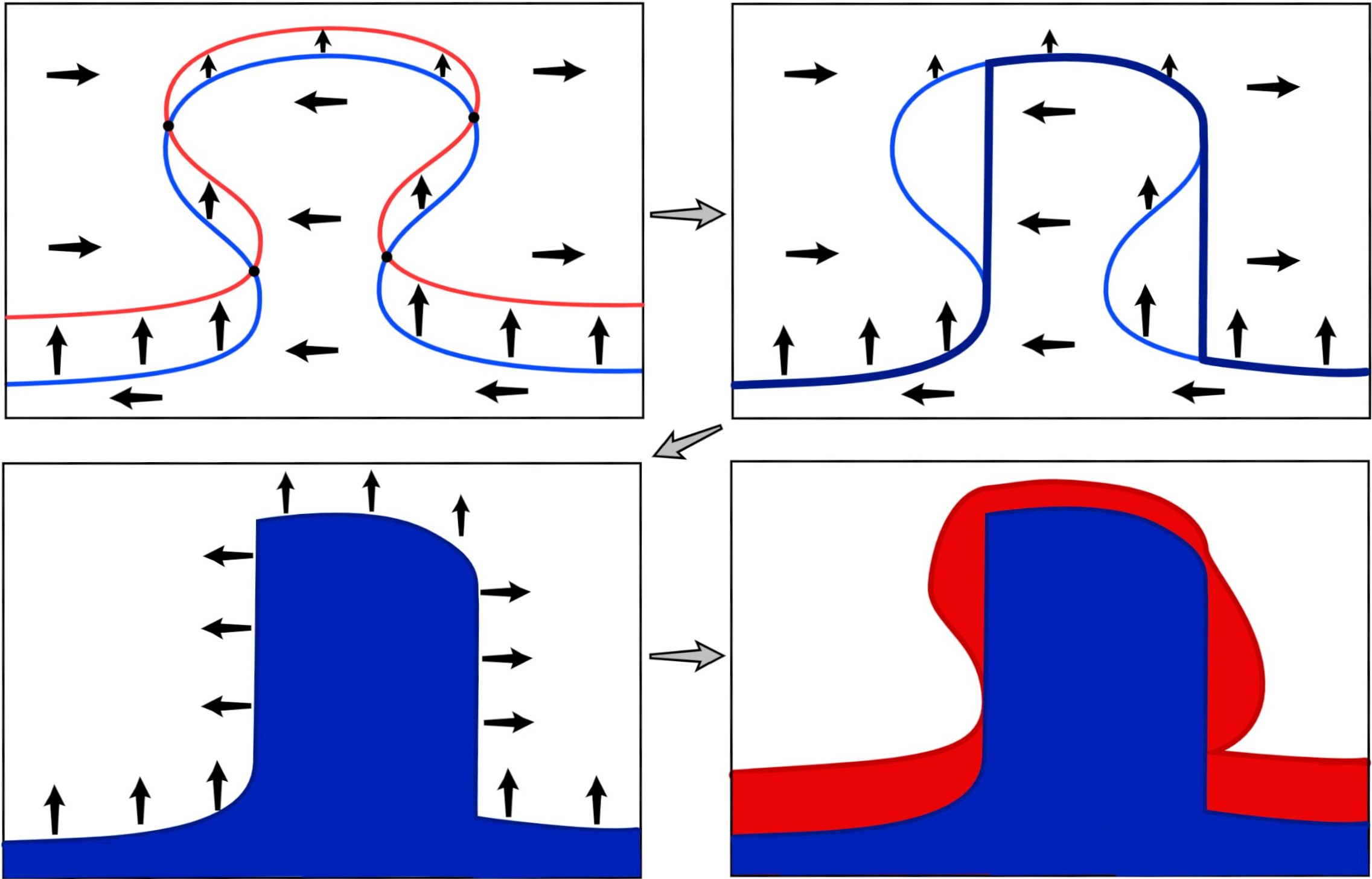}
\caption{\sl An element of $\mathcal{I}$ intersecting its image does not seem to guarantee 
the intersections are fixed points. 
In Panel 1 there are four intersection points, but clearly none of these are fixed points. 
To show this cannot happen, we construct the dark blue curve (Panel 2) which maps outside itself, 
thus contradicting the existence of a positive integral invariant.}
\label{fig:CounterEx}
\end{figure}

Take $S\subseteq \mathcal{I}$, and consider the following two cases:
\begin{enumerate}[label=\alph*.]
    \item For all $(x,y)\in S$, we have $f_2(x,y)>y$.
    \item For all $(x,y)\in S$, we have $f_2(x,y)<y$.
\end{enumerate}

For neither of these conditions to hold, exactly one of the following statements must be true. 
\begin{enumerate}
    \item There exists $(x_1,y_1)\in S$ such that $f_2(x_1,y_1)>y_1$ and $(x_2,y_2)\in S$ such that $f_2(x_2,y_2)<y_2$.
    \item Either for all $(x,y)\in S$ we have $f_2(x,y)\geq y$ or for all $(x,y)\in S$ we have $f_2(x,y)\leq y$ 
    (with equality attained in each case).
\end{enumerate}

In the first case, the intermediate value theorem implies there exist two points satisfying $f_2(x,y)=y$. 
In addition, $(x,y)\in S$ and $f(x,y)$ are in the same fibre of $\alpha$, 
which implies $f(x,y)=(x,y)$. 
Hence, each of these points are fixed points of $f$. 
In the second case, there is a fixed point where the inequality is sharp. 
Therefore we may assume either $f_2(x,y)>y$ or $f_2(x,y)<y$ on all of $S$.

Thus, it is sufficient to show there exists at least one component $S\subseteq \mathcal{I}$ neither 
mapping monotonically upwards nor downwards. 
To accomplish this, 
we assume for the sake of contradiction that each component of $\mathcal{I}$ maps monotonically. 
Then we show that, under this assumption, 
there exists a closed path $C\subset A$ which maps outside itself or inside itself, 
thus violating the existence of a positive integral invariant. 

Without loss of generality, from now on we assume 
that the outer boundary twists clockwise (right) and the inner boundary counterclockwise (left).
 
\subsection{Construction of the path}

Note that if any component $S\subseteq\mathcal{I}$ with $W(S)=\pm1$ is 
free of non-degenerate critical points (see definition below), 
then the region bounded by $S$ violates the existence of a positive integral invariant. 
Thus, for the remainder of this section, we will operate under the assumption that 
such non-degenerate critical points exist for each $S\subseteq\mathcal{I}$ 
(such non-degenerate critical points trivially exist for $S$ with $W(S)=0$). 
Here, we view $A$ as the $1-$periodic strip $(\mathbb{R}/\mathbb{Z})\times[0,1]$, as in Figure \ref{fig:ComplexEx}, rather than the geometric annulus, as in Figure \ref{fig:Intuition1}. Also let $\pi(x,y)=x$. 
As usual we define the $\epsilon-$neighborhood of a set $X\subseteq A$ as $\{(x,y)\in A: \inf\limits_{(x',y')\in X}|(x,y)-(x',y')|< \epsilon \}$.

\begin{itemize}[leftmargin=*]
    \item A point $(x_0,y_0)\in S\subseteq\mathcal{I}$ is a {\bf critical point} 
    if $\dfrac{\partial F_\phi}{\partial y}(x_0,y_0)=0$ (here $S$ is the component of $\mathcal{I}$ containing $(x_0,y_0)$). 
    In other words, the tangent to $S$ at $(x_0,y_0)$ is vertical.
    \item A critical point $(x_0,y_0)\in S\subseteq\mathcal{I}$ is 
    {\bf non-degenerate} if there exists $\epsilon>0$ such that 
    the $\epsilon-$neighborhood $U$ of $S\cap(\{x_0\}\times(0,1))$ 
    satisfies $\pi(U\cap S)\leq x_0$ or $\pi(U\cap S)\geq x_0$. 
    Intuitively, $(x_0,y_0)$ is non-degenerate if the vertical line $x_0\times(0,1)$ does not locally cross $S$.
    \item A non-degenerate critical point is 
    {\bf left-facing} (respectively, {\bf right-facing}) if its tangent is locally to its left (respectively, right). 
    Each of the non-degenerate critical points in Figure \ref{fig:CritPts} are left-facing.
    \item The {\bf outer flow} at a non-degenerate critical point is 
    the direction of the $x-$component of the vector field $f(x,y)-(x,y)$ on the vertical tangent near said critical point.
\end{itemize}

\begin{figure}[H]
\centering
\includegraphics[width=0.5\textwidth]{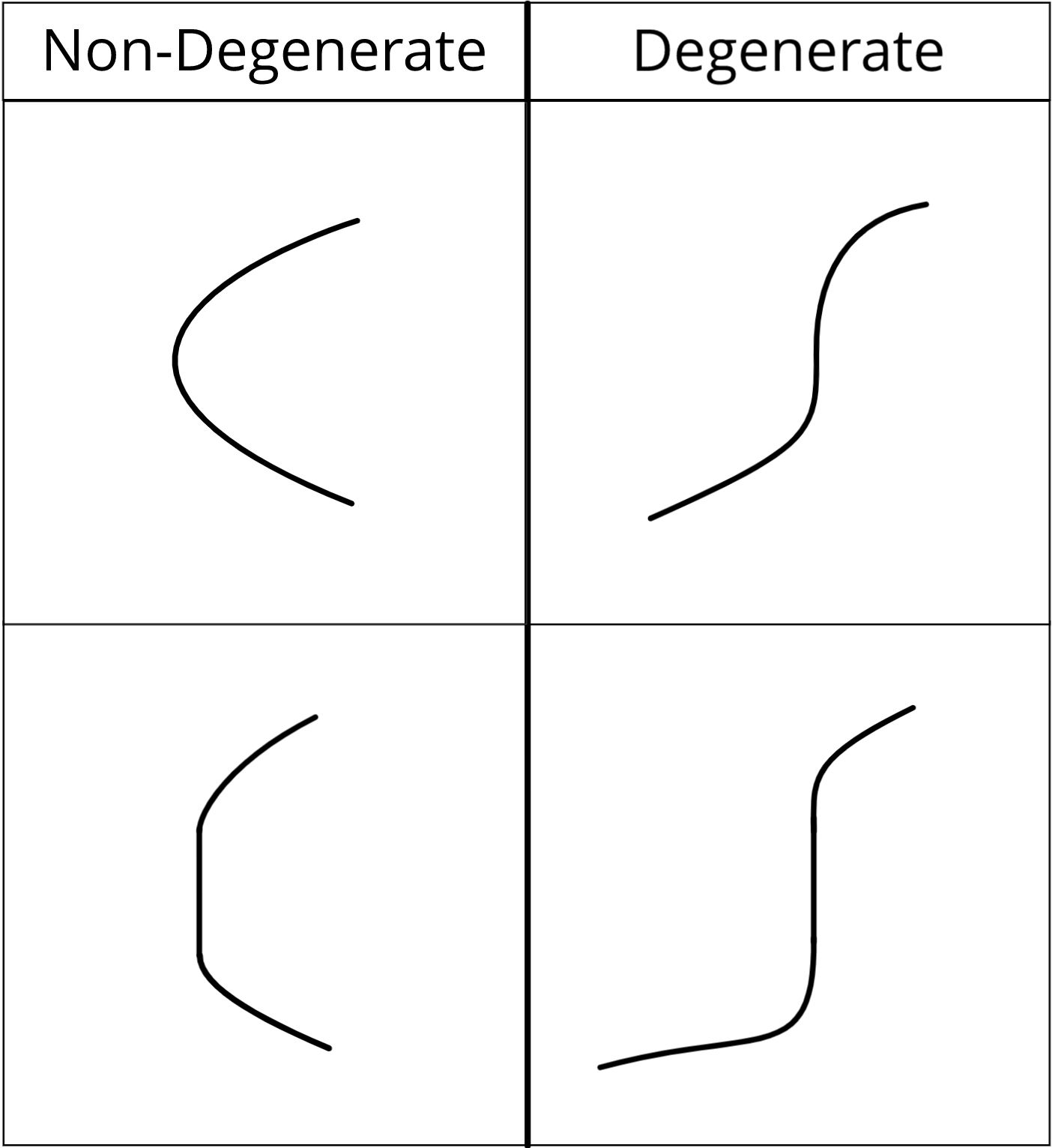}    %%%% FORMELY 0.4 
\caption{\sl Non-degenerate and degenerate critical points. 
Intuitively, non-degenerate critical points are those which 
eventually bend back in the same direction on both sides of the critical point.}
\label{fig:CritPts}
\end{figure}

What follows is a set of conventions which  
will be used to represent and algebraically manipulate the properties of the $S_i$ and the non-degenerate critical points.

\begin{itemize}[leftmargin=*]
    \item Enumerate the $n$ connected components of $\mathcal{I}$ as $\{S_i\}_{i=1}^n$.
    \item For each $S_i$, define the following:
    \begin{itemize}
        \item $u_i$ is the (constant) sign of $f_2(x,y)-y$      
        in $S_i$, denoting the vertical translation direction  of $S_i$ under $f$.
        \item $n_i$ is the number of non-degenerate critical points of $S_i$.
        \item $\{c_{i,k}\}_{k=1}^{n_i}$ is the set of non-degenerate critical points of $S_i$, each equipped with a value:
        \begin{itemize}
            \item $c_{i,k}=+1$ if the non-degenerate critical point is right-facing,
            \item $c_{i,k}=-1$ if the non-degenerate critical point is left-facing.
        \end{itemize}
        \item For 
        $c_{i,k}$, 
        define a corresponding $f_{i,k}$, denoting the direction of the outer flow at $c_{i,k}$:
        \begin{itemize}
            \item $f_{i,k}=+1$ if the outer flow at $c_{i,k}$ is to the right, 
            \item $f_{i,k}=-1$ if the outer flow at $c_{i,k}$ is to the left.
        \end{itemize}
        \item For   
        $c_{i,k}$, define $T(c_{i,k})=u_ic_{i,k}$, called the {\bf ``type'' of $c_{i,k}$}.
    \end{itemize}
\end{itemize}

Now we can begin constructing the closed curve $C$. 
For each non-degenerate critical point $c_{i,k}$ of type $+1$ (that is, with $T(c_{i,k})=u_ic_{i,k}=1$) we define a directed path departing from it as follows.

\begin{enumerate}
    \item Begin traveling vertically in the $v=-f_{i,k}$ direction 
    ({\it i.e.,} 
    $v=+1$ means up and $v=-1$ means down).  
    
    \item Once another component $S_{i'}$ of $\mathcal{I}$ is intersected, travel along it in the $t=-u_{i'}vf_{i,k}$ direction 
    ({\it i.e.,} right if $t=+1$ and left if $t=-1$) 
    until another non-degenerate critical point $c_{i',k'}$ is reached.
\end{enumerate}

\begin{figure}[H]
\centering
\includegraphics[width=1.0\textwidth]{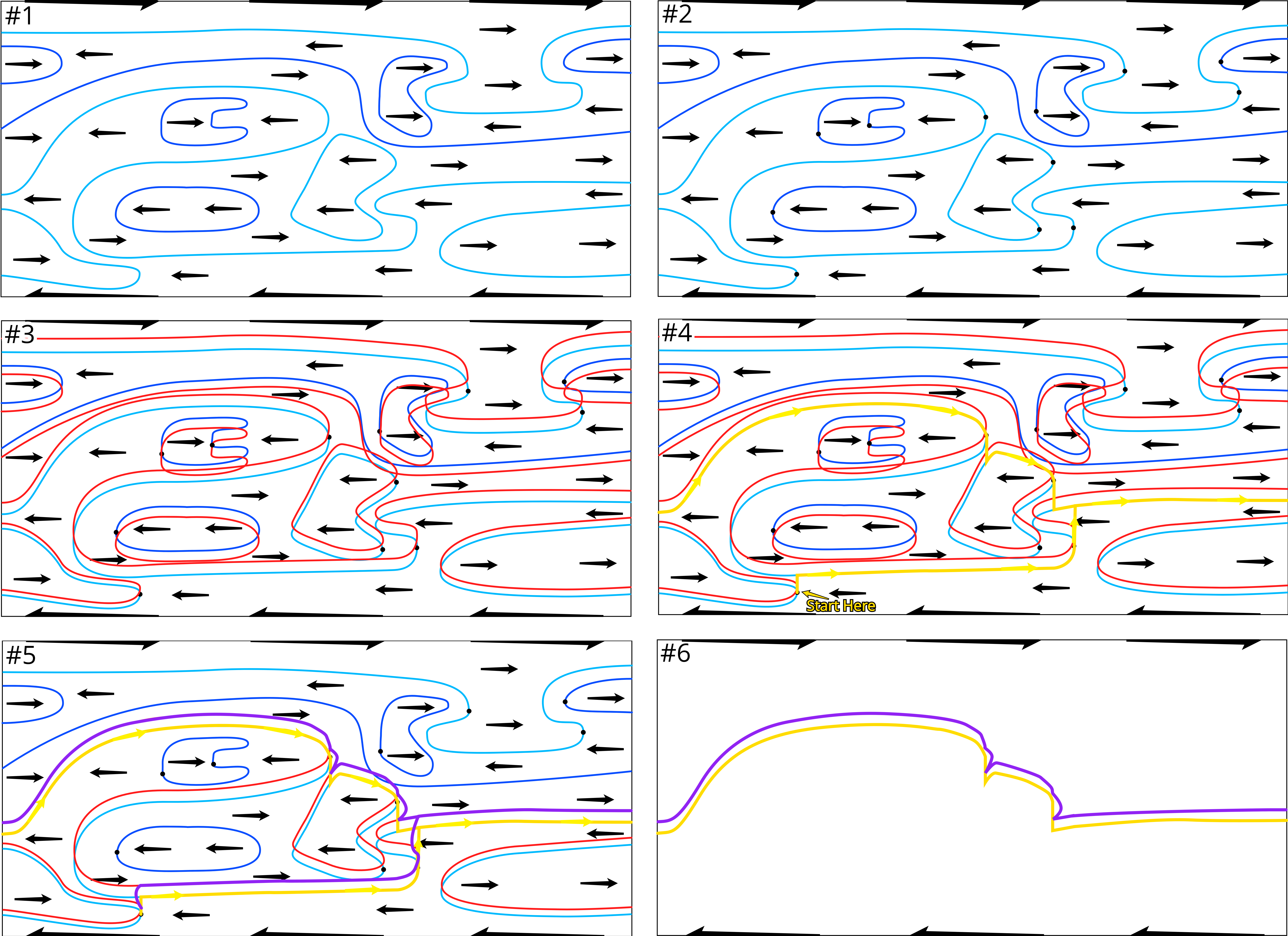}
\caption{\sl A sample application of Lemmas \ref{lemma:PathExistence} and \ref{lemma:LoopBehavior}. 
Panel 1 depicts the set of invariant curves from Figure \ref{fig:ComplexEx}. 
The dark blue curves map downwards and the light blue curves upwards. 
In Panel 2 the type $+1$ non-degenerate critical points are shown in black. 
In Panel 3 the images (red) of the invariant curves are shown. 
In Panel 4 we generate a path (yellow) as described in Lemma \ref{lemma:PathExistence}. 
In Panel 5 we show the image of the generated path (purple). 
Finally, Panel 6 isolates the relevant portion of the generated path (yellow) and its image (purple). 
As we would hope, the path is enclosed by its image, implying a violation of the existence of a positive integral invariant.}
\label{fig:AlgoEx}
\end{figure}

\begin{lemma}\label{lemma:PathExistence}
The path from any $c_{i,k}$ with type $T(c_{i,k})=+1$ exists.  
In fact, 
(1) the vertical part of the path intersects another component and never hits the boundary of $A$,  
and (2) 
the path terminates at a critical point $c_{i',k'}$ that satisfies $T(c_{i',k'})=+1$.
\end{lemma}

Adjacent non-degenerate critical points alternate in direction. 
Thus, because each component of $\mathcal{I}$ has at least one non-degenerate critical point, 
we can conclude that each component has at least one non-degenerate critical point of type $+1$. 
Then, since $\mathcal{I}$ is a compact 1-manifold, 
Lemma \ref{lemma:PathExistence} directly implies there is a finite and non-zero number of non-degenerate critical points of type $+1$.

\begin{lemma}\label{lemma:LoopBehavior}
The collection of directed paths from Lemma \ref{lemma:PathExistence} contains a closed path $C\subset A$. 
Moreover, $f(C)$ never crosses over $C$, and thus violates the existence of a positive integral invariant.
\end{lemma}

Note that this is not the same as requiring that $C$ and $f(C)$ have an empty intersection. We only require that $f(C)$ lies entirely within the closure of one of the two connected components of $A\setminus C$.

The remainder of this section is focused on the proof of 
Lemma \ref{lemma:PathExistence} and Lemma \ref{lemma:LoopBehavior}.

\subsection{Proof of Lemma \ref{lemma:PathExistence}}

Let $c_{i,k}$ be an arbitrary non-degenerate critical point on $S_i\subseteq A$ with $T(c_{i,k})=+1$. 
The proof of Lemma \ref{lemma:PathExistence} can be reduced to three facts.  

\begin{enumerate}
    \item The vertical path taken from $c_{i,k}$ (as defined above Lemma \ref{lemma:PathExistence}) 
    intersects another component $S_{i'}\subseteq\mathcal{I}$ and, in particular, does not intersect $\partial A$.
    \item Once $S_{i'}$ is hit, the continuation along $S_{i'}$ eventually reaches another non-degenerate critical point $c_{i',k'}$.
    \item The terminal non-degenerate critical point $c_{i',k'}$ is of type $+1$.
\end{enumerate}

\subsubsection*{Part 1}

Assume for the sake of contradiction that the vertical portion of the path from some $c_{i,k}$ of type $+1$ intersects the boundary. 
Then the component $S_i$ containing $c_{i,k}$ is either adjacent to the upper boundary or the lower boundary. 
If $S_i$ is adjacent to the lower boundary, then, 
because the lower boundary maps to the left, we have $f_{i,k}=-1$, 
which implies $v=-f_{i,k}=1$. 
Therefore, the vertical portion of the path from $c_{i,k}$ 
cannot intersect the lower boundary. The analogous argument also holds for $S_i$ adjacent to the upper boundary, so this is a contradiction. Thus the vertical taken from $c_{i,k}$ cannot intersect the boundary, and instead intersects another component of $\mathcal{I}$.

\subsubsection*{Part 2}
Because $\mathcal{I}$ is compact, each component has finite length. 
Moreover, we asserted earlier that each component has at least two non-degenerate critical points. 
Therefore, traveling along any component for long enough, we   
eventually reach a non-degenerate critical point.

\subsubsection*{Part 3}

Recall the rules for generating the path from $c_{i,k}$ to $c_{i',k'}$: 
The direction of the vertical is $v=-f_{i,k}$ and the direction along $S_{i'}$ is $t=-u_{i'}vf_{i,k}$. 
Combining these equations gives us the relation $t=-u_{i'}(-f_{i,k})f_{i,k}=u_{i'}$. 
We also have $c_{i',k'}=t$ because if we travel along a component to the right (respectively, left), 
then the non-degenerate critical point we intersect must be right-facing (respectively, left-facing). 
This allows us to directly compute the type of the terminal point of the path: $T(c_{i',k'})=c_{i',k'}u_{i'}=u_{i'}^2=+1$.

\subsection{Proof of Lemma \ref{lemma:LoopBehavior}}

To prove Lemma \ref{lemma:LoopBehavior}, 
it will be sufficient to show that the path does not cross over its image.

\subsubsection*{Such a loop exists}

By Lemma \ref{lemma:PathExistence}, 
the path from any non-degenerate critical point of type $+1$ terminates at another of type $+1$. 
Then, recalling there are finitely many non-degenerate critical points in $\mathcal{I}$, 
any sequence of paths starting on a type $+1$ point must eventually repeat itself, 
thus forming a closed loop $C$.

\subsubsection*{The loop does not cross over its image}

We have, by supposition, that each component of $\mathcal{I}$ maps uniformly either up or down. 
Thus, the only possible trouble points are
\begin{enumerate}[label=\alph*)]
    \item at non-degenerate critical points transitioning to the vertical portion of a path,
    \item on the vertical portions themselves, or  
    \item at the intersection of the vertical part of a path with another component.
\end{enumerate}
\medskip

\noindent 
\textbf{(a) No crossovers when leaving a non-degenerate critical point}\medskip

\noindent 
There are eight ways a type $+1$ non-degenerate critical point may be approached and departed from. 
The cases are listed below.
%\bigskip

\begin{table}[H]
\centering
\small
\tabcolsep=0.11cm
\begin{tabular}{|l|l|l|l|l||l|}
\hline
\multicolumn{6}{|l|}{(1) Direction of Vertical From Critical Point} \\ \hline
\begin{tabular}[c]{@{}l@{}}Case \\ Number\end{tabular} &  \begin{tabular}[c]{@{}l@{}}Component\\Map Direction\end{tabular} & \begin{tabular}[c]{@{}l@{}}Critical Point\\Direction\end{tabular} & \begin{tabular}[c]{@{}l@{}}Outer Flow\\Direction\end{tabular} &
\begin{tabular}[c]{@{}l@{}}Approach\\Direction\end{tabular} &
%\begin{tabular}[c]{@{}l@{}}Vertical Travel\\Direction\end{tabular} \\ \hline
\begin{tabular}[c]{@{}l@{}}Vertical \\Travel \\Direction\end{tabular} \\ \hline
1 & Up & Right & Right & Up & Down \\ \hline
2 & Up & Right & Right & Down & Down \\ \hline
3 & Up & Right & Left & Up & Up \\ \hline
4 & Up & Right & Left & Down & Up \\ \hline
5 & Down & Left & Right & Up & Down \\ \hline
6 & Down & Left & Right & Down & Down \\ \hline
7 & Down & Left & Left & Up & Up \\ \hline
8 & Down & Left & Left & Down & Up \\ \hline
\end{tabular}
\end{table}

The easiest way to check these cases is pictorially, as in Figure \ref{fig:CritCases} below.

\begin{figure}[H]
\centering
\includegraphics[width=\textwidth]{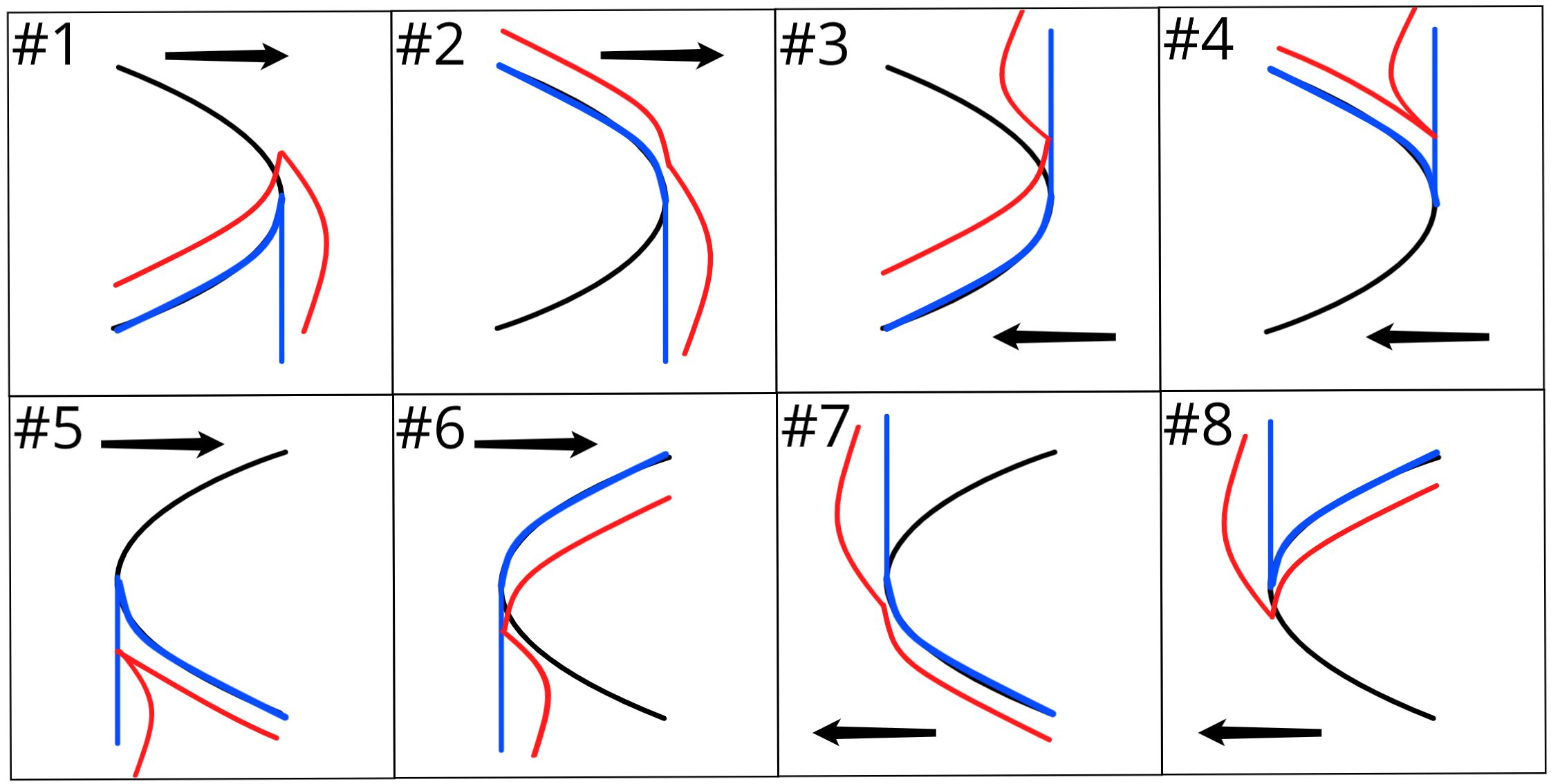}
\caption{\sl Casework for the path when leaving a non-degenerate critical point, 
with the path (blue), its image (red), and 
the path 
$S_i$ (black). 
The arrow in each frame denotes the direction of the outer flow.}
\label{fig:CritCases}
\end{figure}

The path maps monotonically either up or down when restricted to any given component of $\mathcal{I}$, 
and the outer flow direction determines to which side of itself the vertical portion of the path is mapped. 
As we can see, in each of these cases the image of the path does not cross over (but may intersect) the path.
\medskip

\noindent 
\textbf{(b) No crossovers on the vertical portions themselves}   
\medskip

\noindent 
In this case, each vertical portion of the path is confined to a region bound by $\partial A\cup\mathcal{I}$. 
Moreover, $\mathcal{I}$ is the set of points for which the horizontal component $f_1(x)-x$ of the flow $f(x,y)-(x,y)$ changes direction, 
which implies the horizontal component of the flow has constant direction in each of these regions. 
Therefore, each vertical component of the path is mapped strictly either to its left or its right. 
So, the image of these verticals does not intersect the original vertical, except perhaps at its endpoints, and thus does not cross it.
\eject% \medskip

\noindent 
\textbf{(c) No problems at the intersection of a vertical path with another component}
\medskip

\noindent 
As in Part (a), we  show this by casework, tabulating the set of possible situations.

\begin{table}[H]
\centering
\small
\tabcolsep=0.11cm
\begin{tabular}{|l|l|l|l||l|}
\hline
\multicolumn{5}{|l|}{Direction of Travel From Vertical Intersection} \\ \hline
\begin{tabular}[c]{@{}l@{}}Case \\ Number\end{tabular} & \begin{tabular}[c]{@{}l@{}}Component\\Map Direction\end{tabular} & 
\begin{tabular}[c]{@{}l@{}}Vertical Travel\\Direction\end{tabular} & \begin{tabular}[c]{@{}l@{}}Outer Flow\\Direction\end{tabular} & 
\begin{tabular}[c]{@{}l@{}}Travel\\Direction\end{tabular} \\ \hline
1 & Up & Up & Right & Left \\ \hline
2 & Up & Up & Left & Right \\ \hline
3 & Up & Down & Right & Right \\ \hline
4 & Up & Down & Left & Left \\ \hline
5 & Down & Up & Right & Right \\ \hline
6 & Down & Up & Left & Left \\ \hline
7 & Down & Down & Right & Left \\ \hline
8 & Down & Down & Left & Right \\ \hline
\end{tabular}
\end{table}

These cases are worked out pictorially in Figure \ref{fig:VertCases}. %%%% below.

\begin{figure}[H]
\centering
\includegraphics[width=0.9\textwidth]{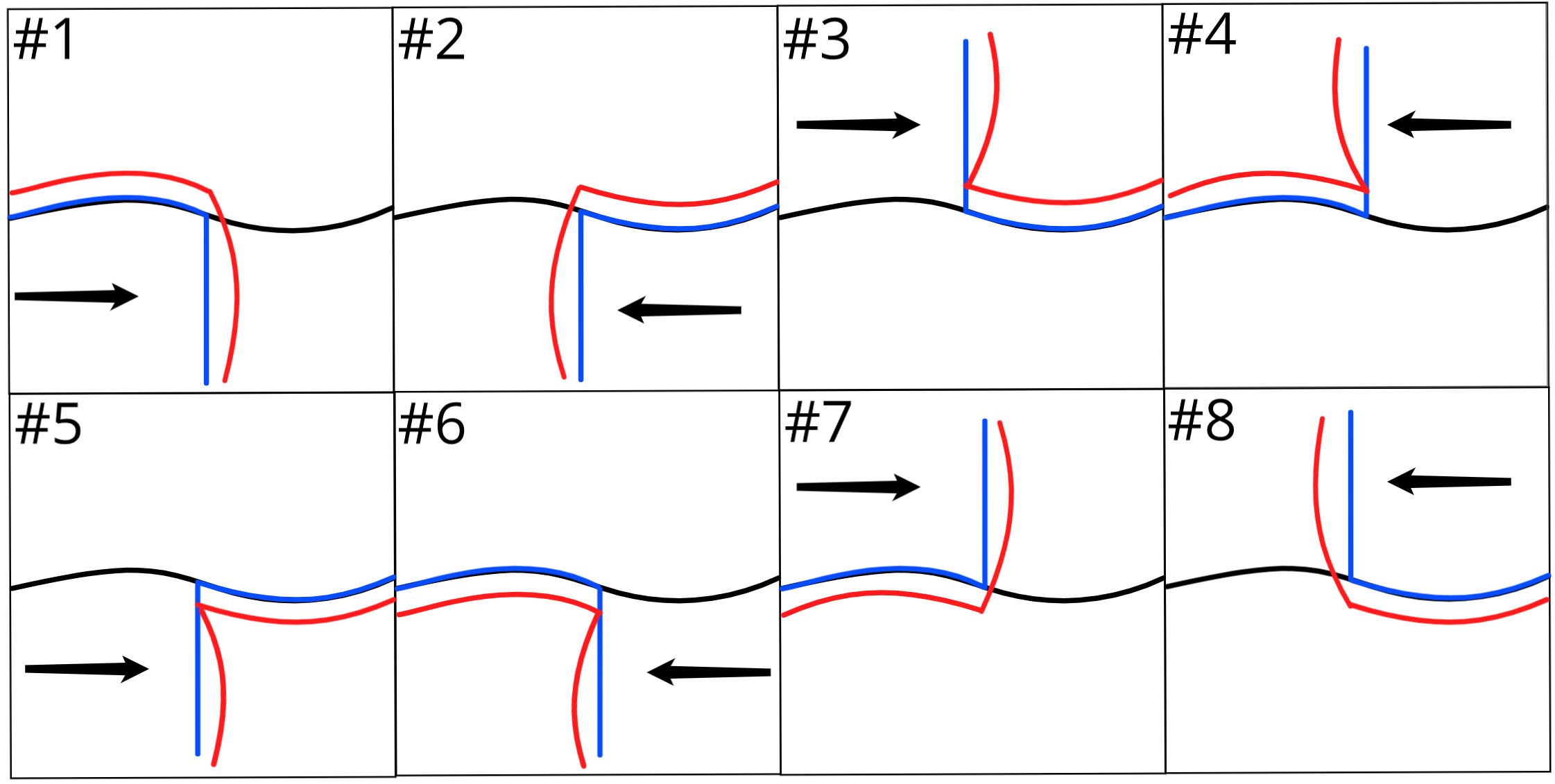}
\caption{\sl Casework for component intersection. 
The elements represented by the colors are the same as in Figure \ref{fig:CritCases}.}
\label{fig:VertCases}
\end{figure}

The portion of the path on the component maps monotonically up or down, 
and the local flow direction determines which side of the vertical the image of the vertical under $f$ is mapped to. 
And as we can see, in each of these cases, the image of the path does not cross over the original, just as in the previous section.
\bigskip

Therefore $f(C)$ never locally crosses over $C$. 
Can $f(C)$ non-locally cross over $C$? 
No, because in this case $f$ would 
not be a diffeomorphism since $f(C)$ 
would either cross over itself or reverse orientation. 
Hence, the region enclosed by $C$ violates the existence of a positive integral invariant under $f$, 
producing a contradiction and proving the existence of at least one fixed point of $f:A\rightarrow A$. 
 %%%% \bigskip 
\end{section}

\begin{section}{There cannot be only one fixed point}
%%%% MOVED PRIOR TO THE EXPLANATION OF HE SECOND FIXED POINT 

The proof that there is a second fixed point is more delicate. 
The approach will be very similar to that given above: 
we assume for the sake of contradiction that there is only one fixed point, 
and use this to produce a contradiction to the existence of a positive integral invariant. 
However, several modifications of the argument are needed to account for the fact that 
the existence of a fixed point prevents $0$ from necessarily being a regular value of 
the same $F$ as before.
%%%% SWITCHED TO SECTION 6

Let $f$ be as in the statement of Theorem \ref{thm:PLGT} and 
having, for the sake of contradiction, a single fixed point $x^\ast$. 
Take a closed ball $B_\epsilon(x^\ast)$ about $x^\ast$, 
with $\epsilon$ sufficiently small so as not to intersect either of the boundaries. 
Next, apply the same construction as in Subsection 4.2, 
except modify the family of functions $F_\phi:A\rightarrow\mathbb{R}$ to 
$F_{\phi,\epsilon}:A\setminus  B_\epsilon(x^\ast)\rightarrow\mathbb{R}$, 
so
$$F_{\phi,\epsilon}=F_{\phi}\big|_{A\setminus B_\epsilon(x^\ast)}.$$
Now $F_{\phi,\epsilon}$ satisfies Theorem \ref{thm:PTT} for all $\epsilon>0$ 
because $A\setminus  B_\epsilon(x^\ast)$ is fixed point-free. 
Thus each connected component of $F_{\phi,\epsilon}^{-1}(0)$ is 
a 1-submanifold of $A\setminus B_\epsilon(x^\ast)$. 
Moreover, as we take $\epsilon$ closer to $0$, $F_{\phi,\epsilon}^{-1}(0)$ is extended towards $x^\ast$. 
At this point, the scenario may look like that in Figure \ref{fig:FPCase1}. 

\begin{figure}[H]
\centering
\includegraphics[width=0.9\textwidth]{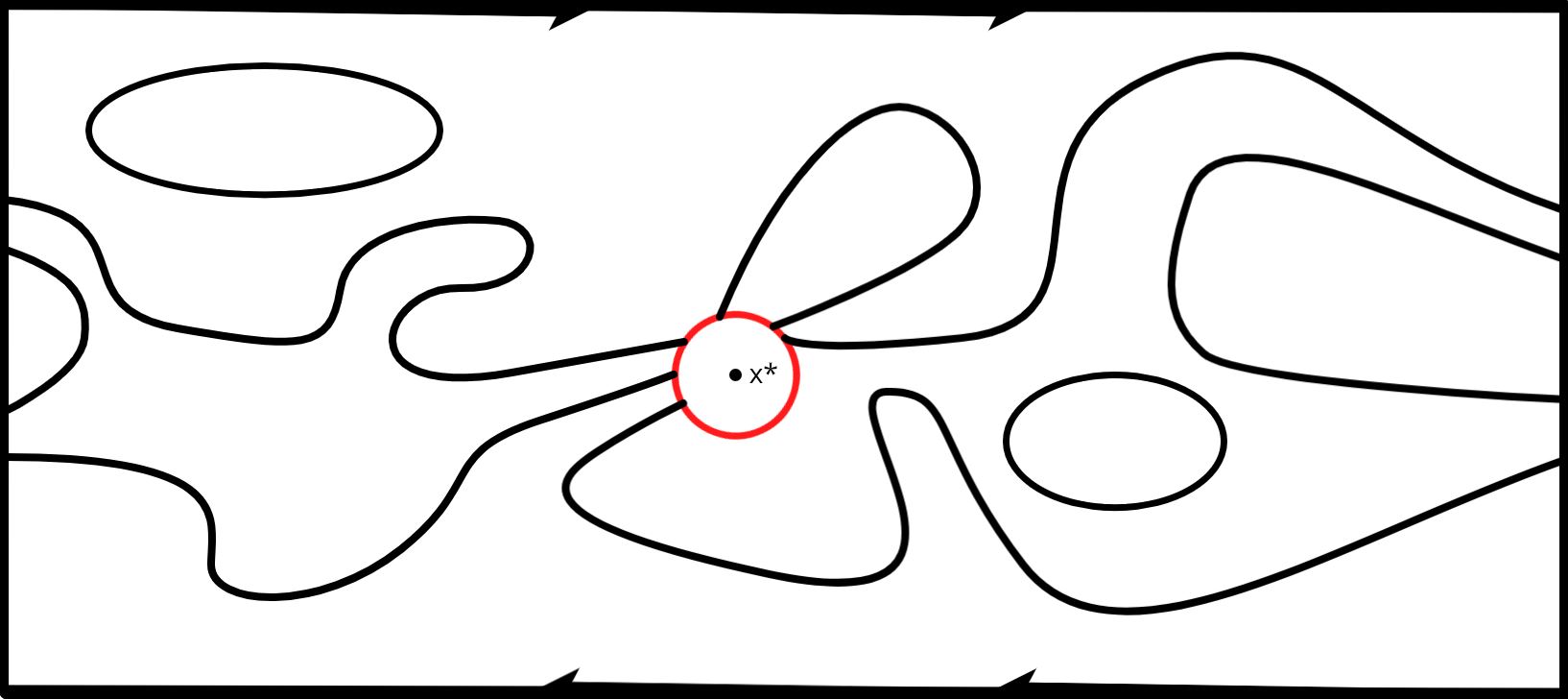}   
\caption{\sl The set of invariant curves of $F_{\phi,\epsilon}^{-1}(0)$: 
the invariant curves (black), the fixed point $x^\ast$, and $\partial B_{\epsilon}(x^\ast)$ (red).}
\label{fig:FPCase1}
\end{figure}

From here, we complete each of the components intersecting $\partial B_\epsilon(x^\ast)$ 
so that we can apply Lemmas \ref{lemma:PathExistence} and \ref{lemma:LoopBehavior}. 
To do so, we append straight line paths from each component to itself. 
For example, a completion of the invariant curves in Figure \ref{fig:FPCase1} 
is shown in Figure \ref{fig:FPCase2}.

\begin{figure}[H]
\centering
\includegraphics[width=0.9\textwidth]{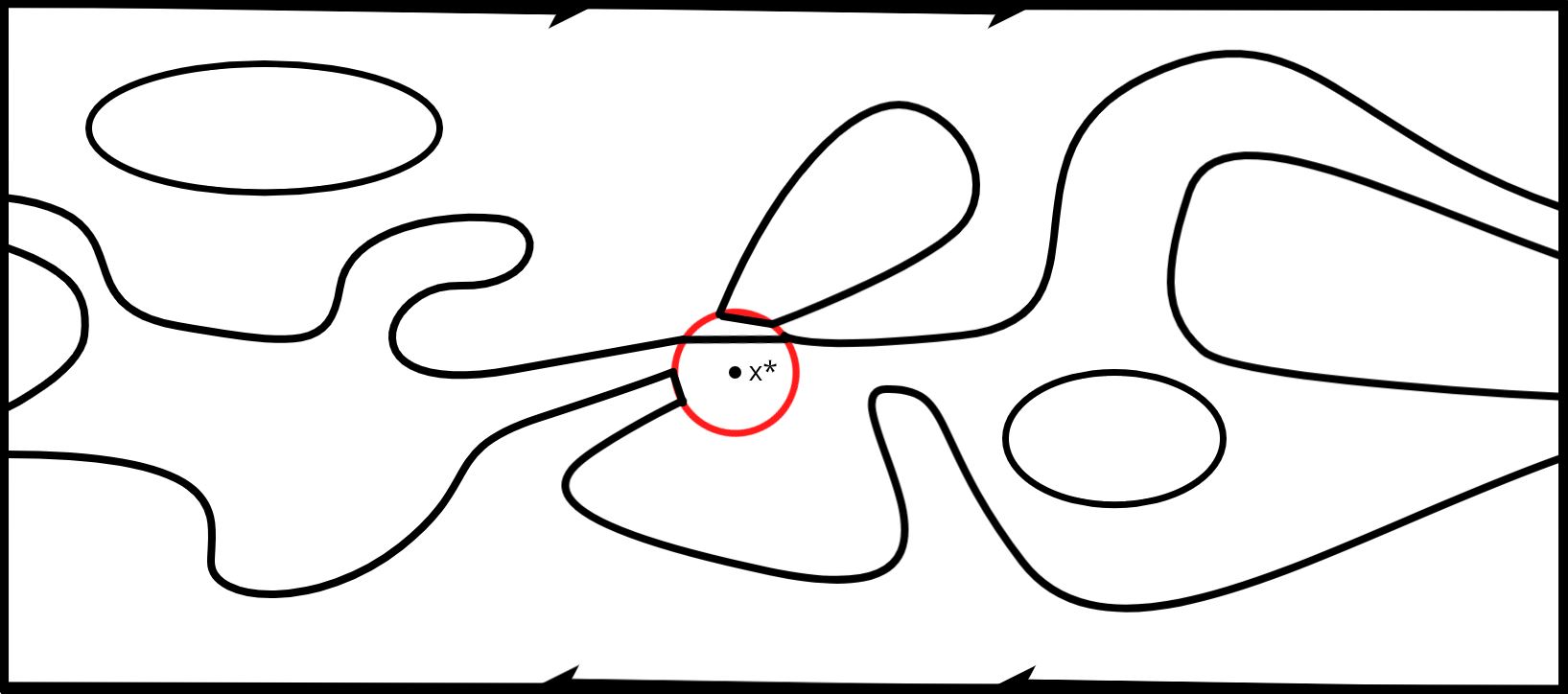}   
\caption{\sl The completion of the set of invariant curves: 
the completed invariant curves (black), the fixed point $x^\ast$, and $\partial B_{\epsilon}(x^\ast)$ (red).}
\label{fig:FPCase2}
\end{figure}
 
For $\epsilon>0$, each component intersects $\partial B_\epsilon(x^\ast)$ 
either twice or not at all because of $\partial F_{\phi,\epsilon}^{-1}(0)=\emptyset$. 
Thus each component is either a closed curve %, 
or has the end points of its closure on $\partial B_\epsilon(x^\ast)$.

Moreover, for $\epsilon>0$ sufficiently small, there exists some component $S$ of $F_{\phi,\epsilon}^{-1}(0)$ 
with winding number different from $0$ which either does not intersect $\partial B_\epsilon(x^\ast)$, 
intersects $\partial B_\epsilon(x^\ast)$ once tangentially, 
or intersects $\partial B_\epsilon(x^\ast)$ twice, 
with the second time being only after winding around $A$ once. 
This follows immediately from the argument given in Section 4.4.

It is important to note that, 
by appending these line segments to each component, 
they are very likely now only $C^0$, rather than $C^1$; 
however, this can clearly only be the case at the (at most) 
finitely many points belonging to 
$\overline{F_{\phi,\epsilon}^{-1}(0)}\cap \partial B_\epsilon(x^\ast)$.

Under this construction, 
Lemma \ref{lemma:PathExistence} still holds and Lemma \ref{lemma:LoopBehavior} 
holds outside $B_\epsilon(x^\ast)$. 
Lemma \ref{lemma:PathExistence} almost works straight out of the box, 
with almost all of the definitions applying the same as before. 
The only place where this fact isn't necessarily clear is in the definition of critical points and non-degenerate critical points. 
It turns out this definition, too, works as originally stated: 
``We say that a point on a component is a critical point if it admits a vertical tangent line''. 
Just to provide additional clarity here, some 
typical examples are exhibited in Figure \ref{fig:C0CritPts} below.
It is also easy to verify that Lemma \ref{lemma:LoopBehavior} continues to hold outside of $B_\epsilon(x^\ast)$.

\begin{figure}[H]
\centering
\includegraphics[width=0.4\textwidth]{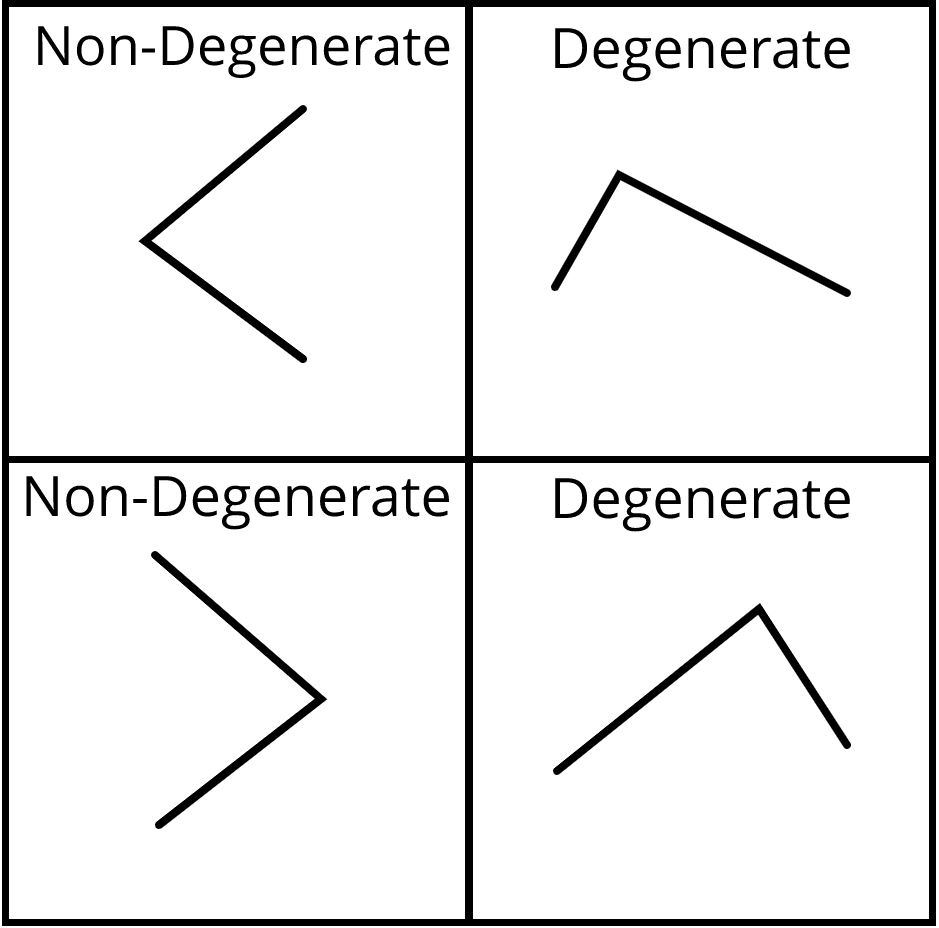}  
\caption{\sl Examples of $C^0$ degenerate and non-degenerate critical points.}
\label{fig:C0CritPts}
\end{figure}

Now all that is left to show is that 
there exists an $\epsilon$ sufficiently small that the measure of the region enclosed by the path $C$ is 
different from the area enclosed by $f(C)$. 
In particular, the amount of error introduced by the failure of Lemma \ref{lemma:LoopBehavior} within $B_\epsilon(x^\ast)$ is 
strictly less than the difference in area enclosed by $C$ and $f(C)$ outside of $B_\epsilon(x^\ast)$. 

To accomplish this, 
we first prove the existence of a lower bound $K$ 
(independent of $\epsilon$)  
on the area differential between $C$ and $f(C)$ outside $B_\epsilon(x^\ast)$. 
Then we  
show that the maximum possible error introduced within $B_\epsilon(x^\ast)$ tends to $0$ as $\epsilon\to0$. 
This would imply we can pick $\epsilon$ sufficiently small so               
that the path is guaranteed to map to a curve which contains strictly more or less area. 
We show these two facts in the remainder of the paper.

\subsection{There exists of a lower bound $K>0$}

We prove the existence of this lower bound in two steps: 
(1) We show that the closed loop from Lemma \ref{lemma:LoopBehavior} has non-zero winding number, and 
(2) we use this fact to construct such a lower bound.

\subsubsection{The path has non-zero winding number}

Suppose for the sake of contradiction that the path generated by Lemma \ref{lemma:LoopBehavior} were to have winding number $0$. 
Then the path may be lifted homeomorphically into  
$\mathbb{R}\times[0,1]$. 
Thus, there must be some region $R$ (a connected component of $A\setminus\mathcal{I}$) in which the path takes at least two verticals. 
Fix either the leftmost or rightmost of these verticals. 
Since there are finitely many critical points, there must exist another vertical of the path in $R$, 
which is horizontally closest to the leftmost (respectively, rightmost) vertical.
Then, by construction, $R$ is composed of points all mapping in the same horizontal direction under $f$. 
Without loss of generality, suppose this direction were to the right and let $C$ be the path containing those verticals. 
There are exactly four distinct cases in which the two verticals can be connected. 
These are shown in Figure \ref{fig:WindingCase}. 

\begin{figure}[H]
\centering
\includegraphics[width=0.6\textwidth]{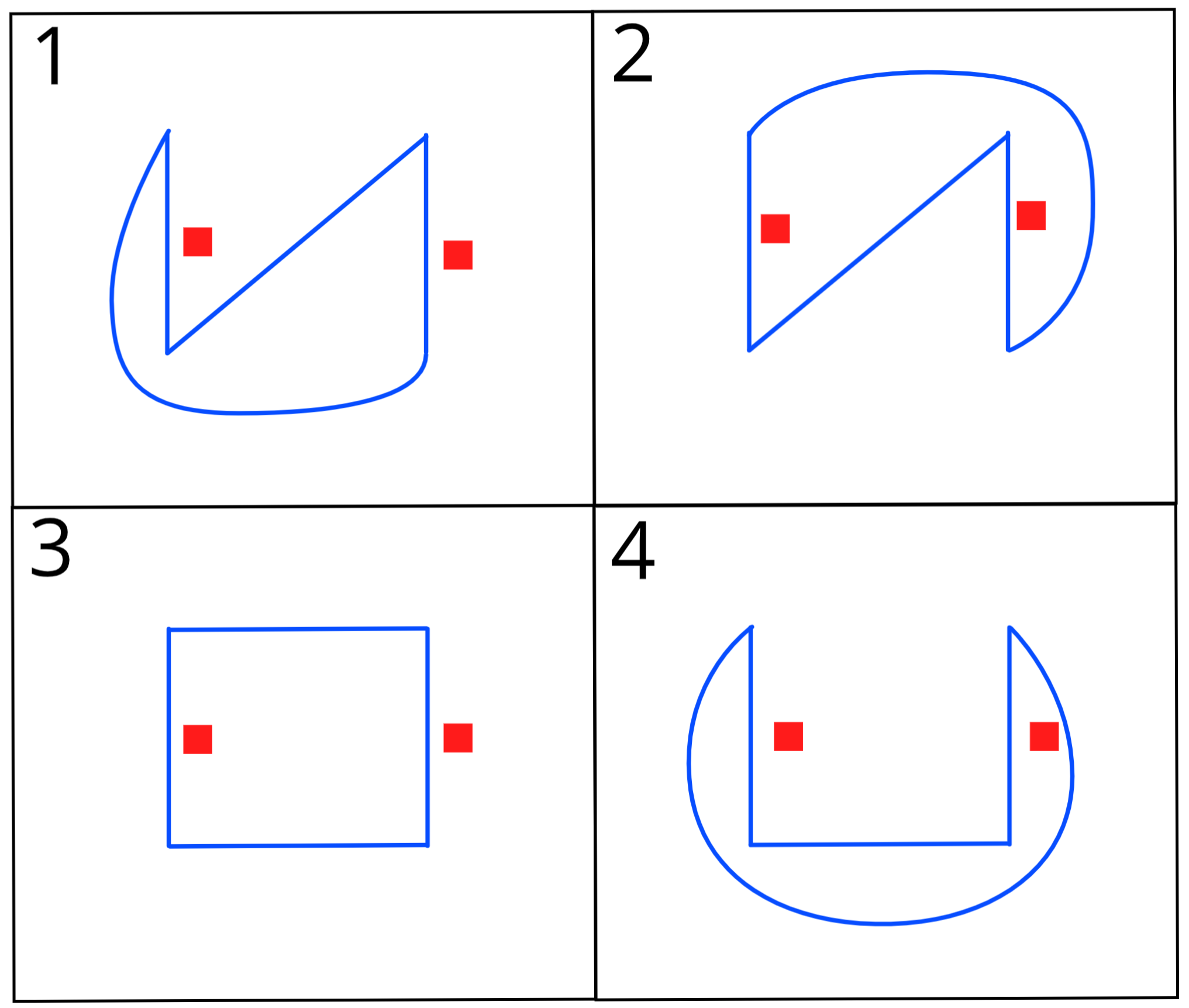}
\caption{\sl All cases in which $C$ (blue) has winding number $0$. 
The boxes (red) indicate the side to which $f$ maps each adjacent vertical.}
\label{fig:WindingCase}
\end{figure}
We can immediately rule out Cases 3 and 4 
because they require that one of the verticals be upward, 
while the other downward. This is impossible 
because, by construction, all verticals in any given region necessarily go in the same direction ($v=-f_{i,k}$). 
On the other hand, in Cases 1 and 2 the diagonal implies the existence of some intermediate vertical in the aforementioned region. 
However, by assumption, no such intermediate vertical exists. 
Therefore, Cases 1 and 2 may be ruled out as well, and we can conclude that 
the path from Lemma \ref{lemma:LoopBehavior} has non-zero winding number.

\subsubsection{The lower bound $K$}

Using the fact that the path from Lemma \ref{lemma:LoopBehavior} 
has non-zero winding number, we obtain the constant $K$ as follows: (1) Find a nonempty subinterval of values of $x$, say $J=[x_{\min},x_{\max}]$, such that $J\times(0,1)\subset A$ contains no non-degenerate critical points, and does not intersect $B_\epsilon(x^\ast)$ 
(such a subinterval must exist because there are only finitely many critical points), (2) consider the area contained between each connected component of $\mathcal{I}\cap(J\times(0,1))$ and their respective images, (3) because the path has non-zero winding number, we know it must pass through at least one of the components above the subinterval $J$, and (4) therefore the least of these areas constitutes such a lower bound $K>0$.

This constant is an effective lower bound on the area enclosed between the path and its image, outside of $B_\epsilon(x^\ast)$. 
Now all that is left is to take $\epsilon$ 
sufficiently small that the lower bound on the area deviation inside $B_\epsilon(x^\ast)$ is less than $K$, 
and we are done.

\subsection{The error arising within $B_\epsilon(x^\ast)$ tends to $0$}

By the compactness of $A$ and the continuity of $Df$, 
the extreme value theorem guarantees there exists $M<\infty$ such that $M\leq|\det(Df(x))|$ for all $x\in A$.

Thus, the area deviation resulting from errors within $B_\epsilon(x^\ast)$ 
is bounded by $\mu(B_\epsilon(x^\ast))\sup_{x\in B_\epsilon(x^\ast)}|\det(Df)|$, 
which is bounded from above by $\mu(B_\epsilon(x^\ast))M$. 
Moreover, 
we have  
$\mu(B_\epsilon(x^\ast))\to0$ as $\epsilon\to0$, 
so we can choose $\epsilon>0$ sufficiently small  
in order to get 
$$\mu(B_\epsilon(x^\ast))<\dfrac{K}{M},$$
which implies
$$\mu(B_\epsilon(x^\ast))M<K.$$
This contradicts the existence of an integral invariant, so we can conclude $f$ has at least two fixed points.
\end{section}

\bibliographystyle{amsplain}

\bibliography{references}

\bigskip
\bigskip

%%%%%%%%%%%%%%%%%%%%%%%%%%%%%%%%%%%%%%%%%%%%%%%%%%%%%%%%
\noindent \textbf{Resumen:} 
Mostramos que el teorema de punto fijo de Poincar\'e-Birkhoff puede ser probado v\'ia una extensi\'on 
del acercamiento geom\'etrico originalmente divisado por el propio Poincar\'e, 
junto con algunos resultados elementales de topolog\'ia diferencial. 
Tras un ejemplo de aplicaci\'on del teorema, 
procedemos a sistem\'aticamente construir y clasificar cierto conjunto de curvas invariantes y sus puntos cr\'iticos. 
Esta clasificaci\'on es luego utilizada para probar la correcci\'on de un procedimiento 
que garantiza la existencia de por lo menos dos puntos fijos de cualquier funci\'on {\it twist} 
de un anillo siempre que admita una integral invariante positiva. 
\bigskip
\bigskip

\noindent \textbf{Palabras clave}: Din\'amica, topolog\'ia diferencial, problema restringido de los tres cuerpos.
\bigskip
\bigskip

\noindent {Andrew Graven} \\
Department of Mathematics, 
Cornell University \\
301 Tower Rd, 
 Ithaca, NY 14853\\
  United States \\
{ajg362@cornell.edu, \quad andrew@graven.com}
\bigskip

\noindent {John Hubbard} \\
Department of Mathematics, 
Cornell University \\
301 Tower Rd, 
 Ithaca, NY 14853\\
  United States \\
{jhh8@cornell.edu}
\medskip

\clearpage 

\thispagestyle{empty}
\cleardoublepage

\end{document}